\documentclass[10pt]{amsart}
\usepackage{amsmath,amssymb,amsthm,graphicx,cite}
\usepackage{appendix}

\usepackage[colorlinks=true,urlcolor=blue,
citecolor=red,linkcolor=blue,linktocpage,pdfpagelabels,bookmarksnumbered,bookmarksopen]{hyperref}
\usepackage[hyperpageref]{backref}
\usepackage{color,mathrsfs}
\usepackage[english]{babel}
\usepackage[left=2.65cm,right=2.65cm,top=2.65cm,bottom=2.65cm]{geometry}

\numberwithin{equation}{section}
\newtheorem{theorem}{Theorem}[section]
\newtheorem{proposition}[theorem]{Proposition}
\newtheorem{lemma}[theorem]{Lemma}
\newtheorem{remark}[theorem]{Remark}

\theoremstyle{definition}

\newcommand{\set}[1]{\left\{#1\right\}}

\newcommand{\R}{\mathbb R}

\newcommand{\C}{\mathbb C}

\newcommand{\EE}{{\mathscr E}}

\newcommand{\eps}{\varepsilon}
\newcommand{\im}{\imath}

\newcommand{\CC}{{\mathbb C}}
\newcommand{\const}{{\rm const}}
\renewcommand{\H}{{\mathcal H}}

\title[Soliton dynamics of NLS with singular potentials]{Soliton dynamics of NLS with singular potentials}

\author[C.\ Bonanno]{Claudio Bonanno}
\author[M.\ Ghimenti]{Marco Ghimenti}
\author[M.\ Squassina]{Marco Squassina}

\address{Dipartimento di Matematica Applicata
\newline\indent
Universit\`a degli Studi di Pisa
\newline\indent
Via F. Buonarroti, 1/c, I-56127 Pisa, Italy}
\email{bonanno@mail.dm.unipi.it}
\email{ghimenti@mail.dm.unipi.it}

\address{Dipartimento di Informatica
\newline\indent
Universit\`a degli Studi di Verona
\newline\indent
C\'a Vignal 2, Strada Le Grazie 15, I-37134 Verona, Italy}
\email{marco.squassina@univr.it}

\thanks{The authors were partially supported by the 
MIUR projects PRIN2009: {\em ``Variational and Topological Methods in the Study of Nonlinear Phenomena''}
and {\em ``Critical Point Theory and Perturbative Methods for Nonlinear Differential Equations''}}


\begin{document}

\begin{abstract}
We investigate the validity of a soliton dynamics behavior in the semi-relativistic limit for the nonlinear Schr\"odinger equation in $\R^{N}, N\ge 3$, in presence of a singular external potential.
\end{abstract}

\maketitle

%%%%%%%%%%%%%%%%%%%%%%%%%%%%%%%%%%%%%%

%\bigskip
%\setcounter{tocdepth}{1}
%\begin{center}
%\begin{minipage}{8cm}
%\footnotesize
%\tableofcontents
%\end{minipage}
%\end{center}
%\medskip
%%%%%%%%%%%%%%%%%%%%%%%%%%%%%%%%%%%%%%

\section{Introduction and main result}
\noindent
For $\eps \in (0,1]$, $N\ge 3$ and $0<p<2/N$, we consider the nonlinear Schr\"odinger equation
\begin{equation}
	\label{problem}
\imath \eps \partial_t u_\eps+\frac{\eps^2}{2}\Delta u_\eps-V(x)u_\eps+|u_\eps|^{2p}u_\eps=0,\qquad t>0,\,\, x\in\R^N
\end{equation}
in presence of a real external potential $V$. This equation typically appears for the 
propagation of light in nonlinear optical materials which exhibit some kind of inhomogeneities, 
see \cite{sulem} and the references therein for more details.  
For a smooth potential $V$, the problem of orbital stability of standing wave solutions to~\eqref{problem}
has been extensively studied, see e.g.~\cite{abc,caze,cazlio}.
Beside some studies of 
\eqref{problem} in the framework of geometric optics and via suitable perturbation methods \cite{abc}, 
several contributions appeared on the rigorous 
derivation of the soliton dynamics behavior in the semi-relativistic limit $\eps\to 0$ for \eqref{problem} 
with bump-like initial data. Essentially, two rather different approaches are currently available in the literature.
On one hand, the seminal paper by Bronski and Jerrard \cite{bronski}, refined by \cite{keraa}, adopted a 
technique which includes a combination of quantum and classical conservation laws 
with the modulational stability property of ground states
due to Weinstein \cite{wa1,wa2}, see \cite{benghimich,benghimich2,keraa} and the references therein.
On the other hand a different and more geometrical approach was developed in a series of papers
\cite{abou1,frolich1,frolich2,frolich3,holzwo}. Subsequently, based on the first approach,
further developments were achieved for a class of weakly coupled Schr\"odinger systems \cite{sys1,sys2} as well as
for equations with an external electromagnetic field \cite{selvit,magn}.
In all of these manuscripts, the external
potential $V$ is always assumed to be a smooth function on $\R^N$ with bounded derivatives up to order three. For rough and time-dependent potentials see \cite{abou2,abou0}. 

In this paper, we shall derive a soliton dynamics behavior 
for \eqref{problem} in presence of a smooth but singular potential. To our knowledge previous results on this case consider only the one dimensional case, see e.g. \cite{holzwo} and \cite{adno}. 
We shall assume that $V$ satisfies the following conditions:
\noindent
\vskip4pt
\begin{itemize}
\item[(V1)] $V\in C^\infty(\R^N \setminus \set{0}, \R)$ is such that
$$
V(x) \sim |x|^{-\beta},\quad\,\,\,  |\nabla V(x)| \lesssim |x|^{-(\beta+1)},\quad\,\,\, 
\big| \nabla |\nabla V|(x)\big| \lesssim |x|^{-(\beta+2)},\qquad\text{as $|x|\to 0$,}
$$ 
where $0< \beta < 1$;
\vskip3pt
\item[(V2)] $V(x) \ge V_0 =\inf_{\R^N} V>0$ for all $x\in \R^N \setminus \set{0}$ 
and $\frac{|\nabla V(x)|^2}{\sqrt{V(x)-V_0}} \in L^N(\R^N\setminus B(0,1));$ 
\vskip4pt
%(il limite dal basso con $V_0$ serve per avere le soluzioni 
%e le stime in $L^2$ e il decadimento all'infinito serve per 
%avere le stime sull'energia iniziale e nella dimostrazione della Proposizione \ref{viene-2});
\item[(V3)] for each $\delta >0$ it holds $\phi(\delta) <+\infty$, where $\phi:(0,\infty)\to (0,\infty)$ is defined by
\begin{equation} 
	\label{il-phi}
\phi(\delta):= \sum_{|\alpha|=0}^3\, \| D^\alpha V 
\|_{L^\infty(B_\delta)}, \qquad  B_\delta := \R^N \setminus B(0,\delta).
\end{equation}
\end{itemize}
Hence $V$ is bounded away from zero and has only one singularity located,
with no loss of generality, at the origin and is elsewhere
smooth and uniformly bounded together with 
the higher order derivatives up to the order three.
Next, we introduce the initial conditions to be assigned to equation \eqref{problem}. Let $H$ denote 
the \emph{energy space}, that is $H^1(\R^N)$ endowed with the standard norm
\begin{equation*}
	% \label{norma}
\| u \|_H^2 := \int_{\R^{N}} \big( |\nabla u|^2 + |u|^2 \big).
\end{equation*}
We also introduce the $H^{1}_{\eps}$-norm defined on $H$ as
\begin{equation}\label{h1-eps}
\| u \|_{H^{1}_{\eps}}^{2} := \frac{1}{\eps^{N-2}}\, \int_{\R^{N}}\, |\nabla u|^2 + \frac{1}{\eps^{N}}\, \int_{\R^{N}} |u|^2,
\quad\,\,\, u\in H.
\end{equation}
Let $R$ be the positive radial solution to
\begin{equation}\label{eq-R}
-\frac 12 \Delta R(x) + R(x) = R(x)^{2p+1}, \qquad x\in \R^{N}.
\end{equation}
It is well known that $R$ is unique (up to translations) \cite{kwong} and 
exponentially decaying, satisfying
\begin{equation}\label{prop-R}
\lim_{|x|\to +\infty} R(x)\, |x|^{\frac{N-1}{2}} e^{|x|}=\const.
\end{equation}

\noindent
Moreover let $(x_0, \xi_0) \in \R^N \times \R^N$ with $x_0\neq 0$. It is readily seen that there exists 
$\delta = \delta(x_0,\xi_0)>0$ such that the solution $(x(t),\xi(t))$ to the Newtonian system
\begin{equation}
\label{newton}
\begin{cases}
\dot x=\xi, & \\ 
\dot\xi=-\nabla V(x),  & \\
x(0)=x_0,  & \\ 
\xi(0)=\xi_0.
\end{cases}
\end{equation}
is global in time and satisfies 
\begin{equation}
	\label{deltabblow}
\inf\limits_{t} |x(t)|> \delta\, , \quad \sup_{t} |\xi(t)| < \sqrt{|\xi_{0}|^{2} + 2V(x_{0})}.
\end{equation}
This easily follows by the Hamiltonian function for \eqref{newton}, given by
\begin{equation}\label{hamil-newton}
\H(x,\xi) = \frac 12 |\xi|^2 + V(x),\qquad x,\xi\in \R^N.
\end{equation}
\noindent
Let $v_{\eps}(x)$ be a function satisfying:
\vskip5pt
\begin{itemize}
\item[(C1)] $v_\eps(x) \in H$ and is radially symmetric with respect to $x_0$;
\vskip5pt
\item[(C2)] there exist $\gamma>0$ and $(x_0, \xi_0) \in \R^N \times \R^N$ with $x_0\not= 0$ such that
$$
\Big\| v_\eps(x) - R\Big( \frac{x-x_0}{\eps} \Big)  e^{\imath \, \frac{\xi_0 \cdot x}{\eps}} \Big\|^2_{H^1_\eps} < \gamma;
$$

\item[(C3)] for $\delta = \delta(x_0,\xi_0)>0$ as defined in \eqref{deltabblow}, 
there exists $\rho \in (0, |x_{0}|-\delta)$ such that 
$$
\text{supp}\, v_\eps(x) \subset B(x_0,\rho);
$$

\item[(C4)] $\frac{1}{\eps^N} \| v_\eps(x) \|^2_{L^2} = \| R \|^2_{L^2} =: m$.
\end{itemize}
\vskip5pt
\noindent
We are then reduced to study the initial value problem
\begin{equation}\label{ivp}
\left\{ \begin{array}{l}
\imath \eps \partial_t u_\eps+\frac{\eps^2}{2}\Delta u_\eps-V(x)u_\eps+|u_\eps|^{2p}u_\eps=0, \\[0.2cm]
u_{\eps}\in H, \\[0.2cm]
u_{\eps}(0,x) = v_\eps(x),
\end{array} \right.
\end{equation}
where $V$ satisfies (V1)-(V3) and the initial datum $v_{\eps}$ satisfies (C1)-(C4).
Under the above assumptions, \eqref{ivp} admits a global strong solution, that is a function
$$ 
u_{\eps} \in C^{0}(\R, H) \cap C^{1}(\R,H^{-1}),
$$
such that $u_{\eps}(0,x) = v_{\eps}(x)$ and, for all $C^\infty_0(\R^N,\C)$ and $t>0$
$$
\Re\int_{\R^{N}}  \imath \eps \partial_t u_\eps(t,x)\bar\varphi(t,x)-\frac{\eps^2}{2}\nabla u_\eps(t,x)\cdot\nabla\bar\varphi(t,x)-
V(x)u_\eps(t,x)\bar\varphi(t,x)+|u_\eps(t,x)|^{2p}u_\eps(t,x)\bar\varphi(t,x)   = 0.
$$
Furthermore, there holds $u_\eps(t) \in H^2(\R^N)$ and $\partial_t u_\eps(t)\in L^2(\R^N)$, for all $t>0$.
Since under our assumptions $V\in L^m(\R^N)+L^\infty(\R^N)$ for $m\geq 2$ with $m>N/2$, 
this holds true in light of \cite[see Theorem 4.3.1 and Remark 4.3.2 for 
local well-posedness and conservation laws as well as 
Theorem 5.2.1 and Remark 5.2.9 for the regularity $H^2(\R^N)$]{caze} jointly with the 
a priori estimate for all $t>0$ obtained in Lemma~\ref{stime-h1}.
\vskip8pt
\noindent
To our knowledge, the following result is the first attempt to describe the soliton dynamics in presence 
of a singular potential in several dimensions. Under the previous assumptions it holds

\begin{theorem}
\label{main}
Assume that, for a small $\eps>0$, we have
\begin{equation}
	\label{cond-small}
	\gamma\leq \eps^{4\frac{17+\beta}{1-\beta}},\qquad
	|\xi_0|\leq \eps^{\frac{17+\beta}{1-\beta}},\qquad
	\int_{B(x_0,\rho)} (V(x)-V_0) |v_{\eps}(x)|^{2}\le \eps^{N+2\frac{17+\beta}{1-\beta}}.
\end{equation}
Then there exists a map $\theta_{\eps}:\R^+\to [0,2\pi)$ such that
\begin{equation}
\label{mainconclus}
u_{\eps}(x,t)=R\Big(\frac{x-x(t)}{\eps}\Big)
e^{\frac{{\imath}}{\eps}\left[x\cdot  \xi(t)+\theta_{\eps}(t)\right]}+\omega_\eps(x,t),   
\end{equation}
locally uniformly in time and $\|\omega_\eps(\cdot, t)\|_{H^1_\eps}
\leq \Gamma\eps$, for some positive constant $\Gamma$. 
\end{theorem}

\noindent
Roughly speaking, in order to preserve the shape of the initial profile and to describe the dynamics, 
one has to start with a bump-like initial data
located sufficiently far from the singularity and with a small enough initial velocity. Precisely, for
the model potential $V(x)=|x|^{-\beta}$ one should assume that $|x_0|\geq  2/\eps^{2(17+\beta)/(\beta-\beta^2)}$
in order to fulfill the last inequality of assumption \eqref{cond-small}.

\noindent
The result is proved by arguments in the spirit of \cite{bronski}. On the other hand, the presence of the singular potential
requires a very careful analysis and new subtle estimates have to be established.
In particular, we refer the reader to Propositions \ref{viene-1} and \ref{viene-2}. 

\noindent
Finally, in Appendix~\ref{semi-sing-sect} we shall provide the estimates related with the soliton
dynamics when the singular potential is truncated around the singularity. We believe that
this can be useful, especially for numerical purposes.

\bigskip

\noindent
Throughout the manuscript we shall always give the explicit dependence of the constants involved in the 
estimates. The constants will often depend on the initial conditions $(x_{0},\xi_{0},v_{\eps})$ 
but in a uniform manner with respect to $\eps$. That is, let $\eps_0$ be such that 
Theorem~\ref{main} holds for $\eps<\eps_0$. Then the different constants 
$\const(x_{0},\xi_{0},v_{\eps})$ in the following can be bounded by $\const(x_{0},\xi_{0}, v_{\eps_0})$.
\medskip

\section{Some preparatory results} \label{sec:solution}

\noindent
Using the variational structure of \eqref{problem}, it is readily checked that the solution $u_{\eps}$ satisfies
\begin{equation}\label{utile-1}
\frac{d}{dt}\, \frac{1}{\eps^{N}}\, |u_{\eps}(t,x)|^{2}\, = - \nabla \cdot p_{\eps}(t,x), \quad\,\,\,  t>0,\,\,\, x\in\R^N,
\end{equation}
\begin{equation}\label{utile-2}
\frac{d}{dt}\, \int_{\R^{N}}  p_{\eps}(t,x) = - \int_{\R^{N}}\frac{1}{\eps^{N}}|u_{\eps}(t,x)|^{2}\nabla V(x),\quad\,\,\, t>0, 
\end{equation}
where
$$
p_\eps (t,x) := \frac{1}{\eps^{N-1}}\, \Im (\bar{u_\eps}(t,x) \, \nabla u_\eps(t,x)), \quad\,\,\,  (t,x)\in\R\times\R^N,
$$
where $\Im(z)$ denotes the imaginary part of the complex number $z$.

\noindent
Both side terms are finite by assumptions on $u_{\eps}$ and (V1) since $|\nabla V| \in L^{N/2}(\R^N)$.
Notice that, equation \eqref{utile-1} implies the conservation of \textit{mass}, for every $\eps>0$,
$$
m := \frac{1}{\eps^{N}}\, \int_{\R^{N}} |u_{\eps}(t,x)|^{2} =\frac{1}{\eps^{N}}\, \int_{\R^{N}} |v_{\eps}(x)|^{2},
$$
and equation \eqref{utile-2} gives the evolution law for the 
momentum 
$$
P_{\eps}(u_{\eps},t) := \int_{\R^{N}} p_{\eps}(t,x) . 
$$
For a global strong solution to \eqref{ivp} the \textit{energy} defined as follows, is conserved 
\begin{equation}\label{energy-tot}
E_{\eps}(u_{\eps},t) := \frac{1}{2\, \eps^{N-2}}\, \int_{\R^{N}} |\nabla u_{\eps}(t,x)|^{2} - 
\frac{1}{(p+1) \eps^{N}}\, \int_{\R^{N}} |u_{\eps}(t,x)|^{2p+2} + \frac{1}{\eps^{N}} \, \int_{\R^{N}} V(x) |u_{\eps}(t,x)|^{2}.
\end{equation}
\noindent
We recall that the function $R$ is a point of minimum for the energy
\begin{equation}\label{energy-R}
\EE(v) := \frac 12  \int_{\R^{N}} |\nabla v(x)|^2 -\frac{1}{p+1}  \int_{\R^{N}} |v(x)|^{2p+2},
\end{equation}
constrained to the manifold of functions in $H^{1}(\R^{N})$ with fixed $L^{2}$-norm equal to $\sqrt{m}$. Let us denote 
$$
\EE_\eps(v) := \frac{1}{2\eps^{N-2}} \int_{\R^{N}} |\nabla v(x)|^2 - \frac{1}{(p+1)\eps^N} \int_{\R^{N}} |v(x)|^{2p+2}.
$$
\noindent
Then, we have the following
\begin{lemma}\label{stima-en-in-1} 
	Assume that $v_{\eps}$ satisfies assumptions {\rm (C1)-(C4)}. Then there exist $\gamma_{0}>0$ 
	and a positive constant merely depending on $R$ and $\xi_{0}$ such that
\begin{equation}\label{uso-ll}
\Big| \EE_\eps(v_\eps) - \EE_\eps \Big(R\Big( \frac{x-x_0}{\eps} \Big)  e^{\imath \, \frac{\xi_0 \cdot x}{\eps}} \Big) \Big| 
\le \const(R,\xi_{0})\, \sqrt{\gamma},
\end{equation}
for every $\gamma \in (0, \gamma_{0})$ and any $\eps>0$ small.
\end{lemma}

\begin{proof}
We shall use the elementary inequality
\begin{equation}\label{brezis-lieb}
\forall\zeta>0 \quad \exists C_\zeta>0: \quad \big| |a+b|^r - |b|^r \big| \le \zeta  |b|^r + C_\zeta  |a|^r,
\end{equation}
for all $a,b\in \C$ and $r\in(1,\infty)$, where $C_\zeta$ blows up as $\zeta^{1-r}$ as $\zeta$ goes to zero. Indeed, we first write
\begin{align*}
& \Big|\frac{1}{\eps^{N-2}} \int_{\R^{N}} |\nabla v_\eps(x)|^2 - 
\frac{1}{\eps^{N-2}}\int_{\R^{N}} \Big| \nabla \Big[R\Big( \frac{x-x_0}{\eps} \Big) \, e^{\imath \, \frac{\xi_0 \cdot x}{\eps}}\Big] \Big|^2 \Big| \\
&\le \frac{\zeta}{\eps^{N-2}} \int_{\R^{N}} \Big| \nabla \Big[R\Big( \frac{x-x_0}{\eps} \Big) \, 
e^{\imath \, \frac{\xi_0 \cdot x}{\eps}}\Big] \Big|^2 + C_\zeta \Big\| v_\eps(x) - 
R\Big( \frac{x-x_0}{\eps} \Big) \, e^{\imath \, \frac{\xi_0 \cdot x}{\eps}} \Big\|^2_{H^1_\eps} \\
& \le \zeta \Big( \int_{\R^{N}} |\nabla R|^2  + m |\xi_0|^2 \Big) 
+ C_\zeta \, \gamma = O(\sqrt{\gamma}),
\end{align*}
after choosing $\zeta = \sqrt{\gamma}$ and using the asymptotics $C_\zeta \sim \gamma^{-1/2}$ for small $\gamma$. The constant in 
$O(\sqrt{\gamma})$ depends only on $R$ and $\xi_0$. Concerning the second term in the energy $\EE_\eps$, we get
\begin{align*}
& \Big|\frac{1}{\eps^{N}} \int_{\R^{N}} |v_\eps(x)|^{2p+2} - \frac{1}{\eps^{N}}\int_{\R^{N}} 
\Big|R\Big( \frac{x-x_0}{\eps} \Big) \, e^{\imath \, \frac{\xi_0 \cdot x}{\eps}} \big|^{2p+2}  \Big| \\
&\le \frac{\zeta}{\eps^{N}} \int_{\R^{N}} \Big|R\Big( \frac{x-x_0}{\eps} \Big)\Big|^{2p+2}  + 
 \frac{C_\zeta}{\eps^{N}} \int_{\R^{N}} \Big| v_\eps(x) - R\Big( \frac{x-x_0}{\eps} \Big) 
 e^{\imath \, \frac{\xi_0 \cdot x}{\eps}}\Big|^{2p+2} 
\end{align*}
By the Gagliardo-Nirenberg inequality 
\begin{equation}\label{gagl-nir}
\| v \|_{L^q} \le \const(q)\, \|v\|_{L^2}^{1-\frac N2 + \frac Nq}\, \| \nabla v \|_{L^2}^{\frac N2 - \frac Nq},\qquad 2\le q\le 2^*,
\end{equation}
choosing $q=2p+2$, in light of (C2) we obtain
$$
\Big|\frac{1}{\eps^{N}} \int_{\R^{N}} |v_\eps(x)|^{2p+2} - \frac{1}{\eps^{N}}\int_{\R^{N}} 
\Big|R\Big( \frac{x-x_0}{\eps} \Big)  e^{\imath \, \frac{\xi_0 \cdot x}{\eps}} \Big|^{2p+2} 
\Big| \le \zeta \|R\|_{L^{2p+2}}^{2p+2} + C_\zeta \const(2p+2)\, \gamma^{p+1} = O(\gamma),
$$
choosing $\zeta=\gamma$ and using $C_\zeta \sim \gamma^{1-p}$ as $\gamma\to 0$. 
Here the constant in $O(\gamma)$ depends only on $R$ and $p$.
\end{proof}

\begin{lemma}\label{stima-en-in-2}
Assume that $v_{\eps}$ satisfy assumptions {\rm (C1)-(C4)}. Then there exists a positive
constant only depending on $R, x_{0}$ and $\xi_{0}$ such that
\begin{equation}\label{uso-ll2}
\Big| \frac{1}{\eps^N} \int_{\R^{N}} V(x) |v_\eps(x)|^2 - mV(x_{0}) \Big| \le \const(R,x_{0}, \xi_{0})\, (\gamma + \eps^{2}) \phi(\delta),
\end{equation}
for all $\gamma>0$ and $\eps>0$, where $\phi$ is defined in \eqref{il-phi} and 
$\delta= \delta(x_{0}, \xi_{0})$ is defined in~\eqref{deltabblow}.
\end{lemma}

\begin{proof}
We write
\begin{align*}
\frac{1}{\eps^N} \int_{\R^{N}} V(x) |v_\eps(x)|^2  &= \frac{1}{\eps^N} \int_{B(x_0,\rho)} V(x) |v_\eps(x)|^2 
= \int_{B(0,\frac{\rho}{\eps})} V(x_0+\eps y) |v_\eps(x_0+\eps y)|^2 \\
&= \int_{B(0,\frac{\rho}{\eps})} V(x_0) |v_\eps(x_0+\eps y)|^2  + 
\int_{B(0,\frac{\rho}{\eps})} \eps ( \nabla V(x_0)\cdot y) |v_\eps(x_0+\eps y)|^2 \\
& +\int_{B(0,\frac{\rho}{\eps})} \eps^2 (\nabla^2 V(x_0+\eps \omega_\eps(y)y)y \cdot y) |v_\eps(x_0+\eps y)|^2 \\
&\leq m V(x_0) + O( \eps^2 \, \phi(\delta)) \int_{B(0,\frac{\rho}{\eps})} |y|^2 |v_\eps(x_0+\eps y)|^2,
\end{align*}
for some $\omega_\eps(y)\in (0,1)$, where we have used the radial symmetry of $v_\eps(x)$, the definition of $\phi(\delta)$ 
in \eqref{il-phi} and assumptions (V3) and (C4). Moreover, we also have
\begin{align*}
\int_{B(0,\frac{\rho}{\eps})} |y|^2 |v_\eps(x_0+\eps y)|^2  &\le 2 \int_{B(0,\frac{\rho}{\eps})} 
|y|^2 |v_\eps(x_0+\eps y) - R(y)e^{\frac{\imath}{\eps} \xi_0\cdot (x_0+\eps y)} |^2 
+ 2\int_{B(0,\frac{\rho}{\eps})} |y|^2 |R(y)|^2 \\
&\le  \frac{2\rho^{2}}{\eps^{2}}\, 
\Big\| v_\eps(x) - R\Big( \frac{x-x_0}{\eps} \Big) \, e^{\imath \, \frac{\xi_0 \cdot x}{\eps}} 
\Big\|^2_{H^1_\eps} + 2\int_{\R^{N}} |y|^2 |R(y)|^2,
\end{align*}
where the last integral is finite by virtue of~\eqref{prop-R}.
\end{proof}

\noindent
We now state the following uniform bound for the $H^{1}_{\eps}$-norm of solutions.

\begin{lemma}
	\label{stime-h1}
Let $u_{\eps}(t,x)$ be a global strong solution of problem \eqref{ivp}. Then
$$
M(x_{0},\xi_{0},v_{\eps}):= \sup\limits_{t\in \R}\, \| u_{\eps}(t,x) \|_{H^{1}_{\eps}} < +\infty.
$$
\end{lemma}

\begin{proof}
By choosing $q=2p+2$ in \eqref{gagl-nir}, by virtue of the conservation of mass, we obtain
\begin{align*}
\| u_{\eps}(t,\cdot) \|_{L^{2p+2}}^{2p+2} &\le \const(p)\, \|u_{\eps}(t,\cdot) \|_{L^{2}}^{2\left(1-\frac N2 + 
\frac{N}{2p+2} \right)(p+1)}\, \| \nabla u_{\eps}(t,\cdot) \|_{L^{2}}^{2\left(\frac N2 - \frac{N}{2p+2} \right)(p+1)} \\
& = \const(p)\, m^{1+p\left(1-\frac N2\right)}\, \left(\eps^{N}\right)^{1+p
\left(1-\frac N2\right)}\, \Big( \frac{1}{\eps^{N-2}} \, \| \nabla u_{\eps}(t,\cdot) 
\|_{L^{2}}^2 \Big)^{\frac{pN}{2}}\, (\eps^{N-2})^{\frac{pN}{2}} \\
& =\const(p)\, m^{1+p\left(1-\frac N2\right)}\, \eps^{N}\, \Big( \frac{1}{\eps^{N-2}} 
\| \nabla u_{\eps}(t,\cdot) \|_{L^{2}}^2 \Big)^{\frac{pN}{2}},\qquad t>0
\end{align*}
In turn, since $p< \frac 2N$, Young's inequality yields
$$
\frac{1}{(p+1) \eps^{N}}\, \int_{\R^{N}} |u_{\eps}(t,x)|^{2p+2}\leq \frac{1}{4\eps^{N-2}}\, 
\int_{\R^{N}}|\nabla u_{\eps}(t,x)|^{2}+\const(p),\qquad t>0.
$$
Therefore, by the conservation of energy, we can write
\begin{equation*}
E_{\eps}(u_{\eps},0) = E_{\eps}(u_{\eps},t) \ge 
\frac{1}{4\eps^{N-2}}\int_{\R^{N}} |\nabla u_{\eps}(t,x)|^{2}
+ \frac{V_0}{\eps^{N}} \int_{\R^{N}} |u_{\eps}(t,x)|^{2}-\const(p),
\end{equation*}
and the thesis follows by $V_{0}>0$.
\end{proof}

\begin{remark}\rm
By virtue of Lemmas~\ref{stima-en-in-1} and~\ref{stima-en-in-2}, 
the initial energy $E_{\eps}(u_{\eps},0)$ remains uniformly bounded with respect to $\eps>0$.
In turn, we have $\sup_{\eps>0} M(x_{0},\xi_{0},v_{\eps})<+\infty.
$
\end{remark}

\noindent
Introducing now the radial notation
\begin{equation}\label{u-radiale}
u_\eps(t,x) = |u_\eps(t,x)| e^{iS_\eps(t,x)},\quad x\in\R^N,\,\, t>0,
\end{equation}
we write
\begin{equation}
	\label{p-radiale}
p_{\eps}(t,x) = \frac{1}{\eps^{N-1}}\, |u_{\eps}(t,x)|^{2}\, \nabla S_{\eps}(t,x),\quad x\in\R^N,\,\, t>0,
\end{equation}
for the momentum density, and the total energy $E_{\eps}$ can be split into the sum
$$
E_{\eps}(u_{\eps},t)= J_{\eps}(u_{\eps},t) + K_{\eps}(u_{\eps},t),\quad t>0,
$$
where $J_{\eps}$ is the \textit{internal energy} and it is defined as
\begin{equation}\label{int-energy}
J_{\eps}(u_{\eps},t) := \frac{1}{2\, \eps^{N-2}}\, \int_{\R^{N}} \big|\nabla |u_{\eps}(t,x)| \big|^{2} 
- \frac{1}{(p+1)\, \eps^{N}}\, \int_{\R^{N}} |u_{\eps}(t,x)|^{2p+2},\quad t>0,
\end{equation}
and $K_{\eps}$ is the \textit{kinetic energy} and it is defined as
\begin{equation}\label{kin-energy}
K_{\eps}(u_{\eps},t) := \frac{1}{2\, \eps^{N-2}}\, \int_{\R^{N}} |u_{\eps}(t,x)|^{2} 
|\nabla S_{\eps}(t,x)|^{2} + \frac{1}{\eps^{N}} \, \int_{\R^{N}} V(x) |u_{\eps}(t,x)|^{2},\quad t>0.
\end{equation}
\noindent
Then, we have the following

\begin{proposition}\label{import}
Let $u_{\eps}$ be a global strong solution of problem~\eqref{ivp} with energy $E_{\eps}$ as in formula \eqref{energy-tot}. 
Then there exist $\gamma_{0}>0$ and a constant depending only on $R,x_{0}$ and $\xi_{0}$ such that, for all $t>0$, 
\begin{equation}\label{forma-en}
\big| E_\eps(u_\eps,t) -  \EE(R) - m \H(x(t),\xi(t)) \big| \le \const(R,x_{0},\xi_{0})\, (\sqrt{\gamma} + \eps^{2}) \phi(\delta),
\end{equation}
for all $\eps>0$ and $\gamma \in (0,\gamma_{0})$, being $\phi$  
defined in \eqref{il-phi}, $\delta= \delta(x_{0}, \xi_{0})$ defined in 
\eqref{deltabblow}, $\H$ the Hamiltonian function \eqref{hamil-newton} and 
$(x(t),\xi(t))$ the solution to the Newtonian system \eqref{newton}.
Furthermore,
\begin{equation}\label{stima-kin}
0\le K_{\eps}(u_{\eps},t) 
\le \frac 12 m |\xi_{0}|^{2} 
+ \frac{1}{\eps^{N}} \, \int_{\R^{N}} V(x) |v_{\eps}(x)|^{2} + \const(R,\xi_{0}) \sqrt{\gamma},
\end{equation}
for every $t>0$ and for any $\gamma \in (0,\gamma_{0})$.
\end{proposition}

\begin{proof}
By the conservation of the energy $E_{\eps}$ for solutions of \eqref{ivp}, we can write
$$
E_\eps(u_\eps,t) = E_\eps(u_\eps,0) = \EE_\eps(v_\eps) + \frac{1}{\eps^N} \int_{\R^{N}} V(x) |v_\eps(x)|^2, \quad t>0. 
$$
Taking into account
$$
\EE_\eps \Big(R\Big( \frac{x-x_0}{\eps} \Big)  e^{\imath \, \frac{\xi_0 \cdot x}{\eps}} \Big) = \EE(R) + \frac 12 m |\xi_0|^2,
$$
and that $\H(x(t),\xi(t)) = \H(x_{0},\xi_{0})$ for all $t>0$ by the conservation of the Hamiltonian for \eqref{newton},
inequality \eqref{forma-en} follows from Lemma~\ref{stima-en-in-1} and Lemma~\ref{stima-en-in-2}.
To prove~\eqref{stima-kin}, observe that since $\| u_{\eps}(t,\eps\,\cdot) \|_{L^{2}}^{2} =m$ for all $t>0$
and $R$ is a point of constrained minimum for $\EE$ on the $L^2$ sphere or radius $\sqrt{m}$, we get
$$
\EE(R) \leq \EE(|u_\eps(t,\eps,\cdot )|)= J_{\eps}(u_{\eps},t),\quad t>0.
$$
Hence, we get 
\begin{align*}
0\le K_{\eps}(u_{\eps},t) &= E_{\eps}(u_\eps,t) - J_{\eps}(u_{\eps},t) \le E_{\eps}(u_{\eps},0) - \EE(R) \\
&= \EE_\eps(v_\eps(x)) + \frac{1}{\eps^N} \int V(x) |v_\eps(x)|^2 - \EE(R) - \frac 12 m |\xi_{0}|^{2} + \frac 12 m |\xi_{0}|^{2} \\
& \le \frac 12 m |\xi_{0}|^{2} + \frac{1}{\eps^N} \int V(x) |v_\eps(x)|^2 +\const(R,\xi_{0},p) \sqrt{\gamma},\quad t>0,
\end{align*}
by virtue of Lemma \ref{stima-en-in-1}.
\end{proof}

\section{Intermediate proofs} \label{sec:proof}
\noindent
As in \cite{bronski,keraa}, we introduce the auxiliary function
\begin{equation} \label{def-psi}
\Psi^\eps(t,x):= u_\eps(t,x(t)+\eps x) \, e^{-\frac \imath 
\eps \xi(t) \cdot (x(t) + \eps x)},\qquad x\in\R^N, \, t>0,
\end{equation}
which satisfies $\| \Psi^\eps(t,\cdot) \|_{L^2}^2 =m$ for all $t>0$. First of all we notice that
\begin{equation} \label{stima-iniz}
\left\| \Psi^\eps(0,\cdot) - R \right\|_{H^1}^2 \le (3 + 2|\xi_0|^2)\, \gamma,
\end{equation}
which follows from simple computations. For the gradient term
\begin{align*}
& \int_{\R^{N}} \big| \nabla  \Psi^\eps(0,x) - \nabla R(x) \big|^2  \\
\noalign{\vskip2pt}
&= \int_{\R^{N}} \big| (\eps \nabla  u_\eps(0,\eps x+x_0) 
-\imath u_\eps(0,\eps x+x_0) \xi_0) e^{-\frac \imath \eps \xi_0 \cdot (x_0 + \eps x)} - \nabla R(x) \big|^2 \\
\noalign{\vskip2pt}
&= \frac{1}{\eps^N} \int_{\R^{N}} \Big| (\eps \nabla v_\eps(y) -\imath v_\eps(y) \xi_0) 
e^{-\frac \imath \eps \, \xi_0\cdot y} - \eps \nabla \Big[R\Big( \frac{y-x_0}{\eps} \Big)\Big] \Big|^2 \\
\noalign{\vskip2pt}
&= \frac{1}{\eps^N} \int_{\R^{N}} \Big| \eps \nabla v_\eps(y) -\imath v_\eps(y) \xi_0 - \eps 
\nabla \Big( R\Big( \frac{y-x_0}{\eps} \Big) e^{\frac \imath \eps \, 
\xi_0\cdot y}\Big) + \imath R\Big( \frac{y-x_0}{\eps} \Big) \xi_0 e^{\frac \imath \eps \, \xi_0\cdot y}\Big|^2 \\
\noalign{\vskip2pt}
& \le \frac{2}{\eps^N} \int_{\R^{N}} \Big[ \eps^2 \Big| \nabla v_\eps(y) - 
\nabla \Big( R\Big( \frac{y-x_0}{\eps} \Big) 
e^{\frac \imath \eps \, \xi_0\cdot y}\Big) \Big|^2 + |\xi_0|^2 \, 
\Big| v_\eps(y) - R\Big( \frac{y-x_0}{\eps} \Big) 
e^{\frac \imath \eps \, \xi_0\cdot y} \Big|^2\Big] \\ 
\noalign{\vskip2pt}
&< 2 (1+|\xi_0|^2) \gamma,
\end{align*}
where in the last inequality we have used (C2). For the $L^2$ term again
\begin{equation*}
\int_{\R^{N}} \big| \Psi^\eps(0,x) - R(x) \big|^2  
= \frac{1}{\eps^N} \int_{\R^{N}} 
\big| v_\eps(y) - R\Big( \frac{y-x_0}{\eps} \Big)e^{\frac \imath \eps \, \xi_0\cdot y}\big|^2 < \gamma,
\end{equation*}
by virtue of (C2). By definition, it is natural to compute the energy $\EE$ 
defined in \eqref{energy-R} for $\Psi^{\eps}$. We can 
use \eqref{brezis-lieb} as in the proof of Lemma \ref{stima-en-in-1} to obtain
\begin{equation}\label{en-stima-iniz}
0 \le  \EE\left( \Psi^\eps(0,x) \right) - \EE(R) = O(\sqrt{\gamma}),
\end{equation}
where $O(\cdot)$ depends only on $R,x_0,\xi_0$, and we used the 
fact that $R$ is the point of minimum for $\EE$ on the manifold of functions with $L^2$ norm equal to $\sqrt{m}$. 
Moreover, we have the following

\begin{lemma}\label{as-in-k}
There exist $\gamma_{0}>0$ and a positive constant depending only on $R,x_{0}$ and $\xi_{0}$ such that  
\begin{align*}
 0\le \EE\left( \Psi^\eps(t,x) \right) - \EE(R) &\le   m  |\xi(t)|^2 - \xi(t) \cdot 
\int_{\R^{N}} p_\eps(t,x) + m V(x(t)) \\ 
& - \frac{1}{\eps^N} \int_{\R^{N}} V(x) |u_\eps(t,x)|^2   
+ \const(R,x_{0},\xi_{0})\, (\sqrt{\gamma} + \eps^{2}) \phi(\delta),\quad t>0,
\end{align*}
for every $\eps>0$ and $\gamma \in (0,\gamma_{0})$.
\end{lemma}

\begin{proof} 
	The left inequality follows from the properties of 
	$R$ and $\| \Psi^\eps(t,\cdot) \|_{L^2}^2 = m$, for every $t>0$.
Concerning the estimate from above, we use \eqref{u-radiale}-\eqref{p-radiale} to write
\begin{align*}
\EE\left( \Psi^\eps(t,x) \right) &= \frac 12 \int_{\R^{N}} 
|\nabla \Psi^\eps(t,x)|^{2} - \frac{1}{p+1} \int_{\R^{N}}\, |\Psi^\eps(t,x)|^{2p+2} \\
& = \frac 12 \int_{\R^{N}}\, \big|\nabla |u_{\eps}(t,x(t)+\eps x)| \big|^{2} 
+ \frac 12 \int_{\R^{N}}\, |u_{\eps}(t,x(t)+\eps x)|^{2}\, \big| \xi(t) - 
\nabla \left( S_{\eps}(t,x(t)+\eps x) \right) \big|^{2} \\
& - \frac{1}{p+1} \int_{\R^{N}}\, |u_{\eps}(t,x(t)+\eps x)|^{2p+2}  \\
&= J_{\eps}(u_{\eps},t) + K_{\eps}(u_{\eps},t) - \frac{1}{\eps^{N}} \int_{\R^{N}}\, V(x)\, |u_{\eps}(t,x)|^{2} \\ 
& + \frac12 \, m \, |\xi(t)|^{2} - \int_{\R^{N}}\, p_{\eps}(t,x) \cdot \xi(t),\quad t>0,
\end{align*}
where we have used the expressions \eqref{int-energy}-\eqref{kin-energy} for 
the internal and kinetic energy of $u_{\eps}$. Hence, we get
\begin{align*}
 \EE\left( \Psi^\eps(t,x) \right) - \EE(R) &= E_{\eps}(u_{\eps},t) - \EE(R) - m \H(x(t),\xi(t)) \\
&+ m |\xi(t)|^2 - \xi(t) \cdot \int_{\R^{N}} p_\eps(t,x) + m V(x(t)) - 
\frac{1}{\eps^N} \int_{\R^{N}} V(x) |u_\eps(t,x)|^2,\quad t>0.
\end{align*}
The assertion then follows from inequality~\eqref{forma-en} in Proposition~\ref{import}.
\end{proof}

\noindent
Let us now introduce, for any $t>0$, the terms
\begin{equation}\label{gli-eta}
\eta_1^\eps(t):= m \xi(t) - \int_{\R^{N}} p_\eps(t,x), \qquad \eta_2^\eps(t) := m V(x(t)) - \frac{1}{\eps^N} \int_{\R^{N}} V(x)\, |u_\eps(t,x)|^2.
\end{equation}
From Lemma \ref{as-in-k} we have
\begin{equation}\label{pro-wein}
0\le \EE\left( \Psi^\eps(t,x) \right) - \EE(R) 
\le |\xi(t)| |\eta_1^\eps(t)| + |\eta_2^\eps(t)| 
+ \const(R,x_{0},\xi_{0},\delta) (\sqrt{\gamma} + \eps^{2}),
\end{equation}
for every $\eps>0$ and $\gamma \in (0,\gamma_{0})$.
If we write, as in the proof of Lemma \ref{as-in-k},
$$
\EE\left( \Psi^\eps(t,x) \right) = \EE(|  \Psi^\eps(t,x) |) + 
\frac{1}{2 \eps^N} \int_{\R^{N}} |u_\eps(t,x)|^2  | \xi(t) - \eps \nabla S_\eps(t,x)|^2,
$$
from \eqref{en-stima-iniz} and Lemma \ref{as-in-k} we obtain
\begin{align} \label{stime-momento-1}
&\frac{1}{\eps^N} \int_{\R^{N}} |v_\eps(x)|^2  
| \xi_{0}- \eps \nabla S_\eps(0,x)|^2 = O(\sqrt{\gamma}), \\
\label{stime-momento-2}
& \frac{1}{\eps^N} \int_{\R^{N}} |u_\eps(t,x)|^2  | \xi(t) - \eps 
\nabla S_\eps(t,x)|^2  \le |\xi(t)| |\eta_1^\eps(t)| + 
|\eta_2^\eps(t)| + O(\sqrt{\gamma} + \eps^{2}), \quad
t>0,
\end{align}
since $\EE\left( \left|  \Psi^\eps(t,x) \right| \right)  - \EE(R)\ge 0$, for all $t\geq 0$.

\vskip6pt
\noindent
Let us now recall the well-known quantitative property which follows 
from M.\ Weinstein modulational stability theory \cite{wa1,wa2}.

\begin{proposition}
	\label{weinstein}
There exist two positive constants ${\mathcal C}$ and ${\mathcal A}$ such that
$$
\inf_{\substack{\xi\in\R^N \\ \theta\in [0,2\pi)}}\|\Psi-e^{\imath \theta}
R(\cdot-\xi)\|^2_{H^1}\leq {\mathcal C} ({\mathscr E}(\Psi)-{\mathscr E}(R)),
$$
for every $\Psi\in H^1(\R^N)$ such that $\|\Psi\|_{L^2}=\|R\|_{L^2}$ and 
${\mathscr E}(\Psi)-{\mathscr E}(R)<{\mathcal A}$.
\end{proposition}

\noindent
Let us now fix a time $T>0$, $\eps_{0}>0$ as in \eqref{eps-0} 
and $\gamma_{0}>0$ as in Lemma \ref{as-in-k}. Let us set
\begin{equation}\label{t-eps}
T^{\eps,\gamma} := \sup \set{ t \in [0,T]: 
\ |\xi(s)|\, |\eta_1^\eps(s)| + |\eta_2^\eps(s)| \le \mu,\quad \text{for all $s\in (0,t)$}}, 
\end{equation}
where $\mu>0$ is such that 
$$
\mu + \const(R,x_{0},\xi_{0},\delta) 
(\sqrt{\gamma} + \eps^{2})< {\mathcal A},  
\qquad \text{for all $\eps<\eps_{0}$ and $\gamma <\gamma_{0}$},
$$
where $\const(R,x_{0},\xi_{0},\delta)$ is as in~\eqref{pro-wein} and 
${\mathcal A}$ is as in Proposition~\ref{weinstein}, so that 
$\EE\left( \Psi^\eps(t,x) \right) - \EE(R) < {\mathcal A}$ by virtue of \eqref{pro-wein} 
for all $t\in [0,T^{\eps,\gamma})$. Then, in turn, Proposition \ref{weinstein} yields
functions $\varpi^{\eps}:[0,T^{\eps,\gamma})\to[0,2\pi)$
and $w^{\eps}:[0,T^{\eps,\gamma})\to\R^N$ such that
\begin{equation} 
	\label{direct-wein}
\big\|\Psi^{\eps}(t,x)-e^{\imath \varpi^{\eps}(t)}R(x+w^{\eps}(t))
\big\|^2_{H^1}\leq {\mathcal C}(|\xi(t)|\, |\eta_1^\eps(t)| + |\eta_2^\eps(t)| 
+ \const(R,x_{0},\xi_{0},\delta) (\sqrt{\gamma} + \eps^{2})),
\end{equation}
for all $t\in [0,T^{\eps,\gamma})$. Then, we get the following

\begin{lemma}\label{rappr}
There exist families of functions $\theta^{\eps}:[0,T^{\eps,\gamma})\to[0,2\pi)$ and $x^{\eps}: [0,T^{\eps,\gamma})\to\R^N$ such that
$$
\big\|u_\eps(t,x)-e^{\frac{\imath}{\eps}(\xi(t)\cdot x+\theta^{\eps}(t))}
R\Big(\frac{x-x^\eps(t)}{\eps}\Big)\big\|^2_{H^1_\eps}\leq {\mathcal C}(|\xi(t)|\, |\eta_1^\eps(t)| + |\eta_2^\eps(t)| + 
\const(R,x_{0},\xi_{0})\, (\sqrt{\gamma} + \eps^{2}))
$$
for all $t\in [0,T^{\eps,\gamma})$. 
\end{lemma}

\begin{proof}
In light of inequality \eqref{direct-wein}, defining the functions $\theta^{\eps}:[0,T^{\eps,\gamma})\to[0,2\pi)$
and $x^{\eps}:[0,T^{\eps,\gamma})\to\R^N$ by setting $\theta^\eps(t):=\eps \varpi^\eps(t)$ 
and $x^\eps(t):=x(t)-\eps w^\eps(t)$ for every $[0,T^{\eps,\gamma})$
respectively, the assertion follows by the definition of $\Psi^\eps$.
\end{proof}

\noindent
We now consider the behavior of the difference $|x^{\eps}(t) - x(t)|$. This can be done as in \cite{keraa}, since the proofs do not depend on the properties of the potential $V$. Let $\chi$ denote the cuff-off function which is defined in \cite[p.179]{keraa}. Then we can get

\begin{lemma}
       \label{w-come-in-k}
For every $t\in [0, T^{\eps,\gamma})$ we have
\begin{equation}
	\label{laprimacosa}
\eps |w^{\eps}(t)|=|x^{\eps}(t) - x(t)|\le \const(R,x_{0},\xi_{0},\delta) (|\eta_{1}^{\eps}(t)| + |\eta_{2}^{\eps}(t)|+ |\eta_{3}^{\eps}(t)| 
+\sqrt{\gamma} + \eps^{2}),  
\end{equation}
where $\eta_3^\eps(t)$ is defined as 
$
\eta_{3}^{\eps}(t):=\frac{1}{\eps^N}\int_{\R^N} x\chi(x)|u_\eps(t,x)|^2- m x(t)
$
and it satisfies
\begin{equation}
	\label{lasecondacosa}
\eta_{3}^{\eps}(0) \leq\const(R,x_{0},\xi_{0},\delta)\eps^2, \quad \Big|\frac{d}{dt}\eta_3^\eps(t) \Big| \le 
 \const(R,x_{0},\xi_{0},\delta) (|\eta_{1}^{\eps}(t)| + |\eta_{2}^{\eps}(t)|+ |\eta_{3}^{\eps}(t)| 
+\sqrt{\gamma} + \eps^{2}).
\end{equation}
\end{lemma}
\begin{proof}
The proof of \eqref{laprimacosa} follows by just mimicking step by step the proof of 
\cite[Lemma 3.5]{keraa}, which is based on the arguments of \cite[Lemma 3.4]{keraa} 
in view of our inequalities \eqref{pro-wein}-\eqref{direct-wein}. 
Notice also that in this proof one needs to choose the time $T$ properly, 
but depending only on $x_{0}, \xi_{0}, \eps_{0}, \gamma_{0}$ and ${\mathcal A}$. 
This is analogous to \cite[Lemma 3.4]{keraa}.
Instead, concerning properties \eqref{lasecondacosa} it is sufficient to argue as in \cite[Lemma 3.6]{keraa}.
\end{proof}

\noindent
We now redefine the time $T^{\eps,\gamma}$ by also imposing $w^{\eps}$ to be bounded. Namely
\begin{equation}\label{t-eps-new}
T^{\eps,\gamma} := \sup \set{ t \in [0,T]: 
\ |\xi(s)|\, |\eta_1^\eps(s)| + |\eta_2^\eps(s)| \le \mu, \, \text{and }\, |w^{\eps}(s)|\le 1, 
\quad \text{for all $s\in (0,t)$}}
\end{equation}

\noindent
The last ingredients for the proof of the main result are estimates for the behavior of the quantities $\eta_1^\eps$ and 
$\eta_2^\eps$ defined in \eqref{gli-eta} in the interval $[0,T^{\eps,\gamma})$. It follows that these quantities have time derivatives bounded by
\begin{equation}\label{eta-tutti}
|\eta^{\eps}(t)| := |\eta^{\eps}_{1}(t)| + |\eta^{\eps}_{2}(t)| + |\eta^{\eps}_{3}(t)|
\end{equation}
up to an error depending on the kinetic energy $K_{\eps}(u_{\eps},t)$ and on terms of the order $\sqrt{\gamma} + \eps^{2}$.

\begin{lemma}
	\label{eta-iniz} 
There exists positive constants only depending on $R$, $x_{0}$ and $\xi_{0}$ such that
$$
|\eta_1^\eps(0)| \le \const(R,x_{0},\xi_{0}) \gamma^\frac 14, 
\qquad |\eta_2^\eps(0)| \le \const(R,x_{0},\xi_{0}) (\gamma + \eps^2).
$$
\end{lemma}

\begin{proof} 
Let us recall the radial notation~\eqref{p-radiale} 
for the momentum density. Then, we write
\begin{align*}	
|\eta_1^\eps(0)| & = \Big| m\xi_0 - \int_{\R^{N}} p_\eps(0,x) \Big| = \Big| \frac{1}{\eps^{N-1}} \int_{\R^{N}} 
R^2\Big( \frac{x-x_0}{\eps} \Big) \frac{\xi_0}{\eps} - \frac{1}{\eps^{N-1}} \int_{\R^{N}} |v_\eps(x)|^2 \nabla S_\eps(0,x) \Big|  \\
&= \Big| \frac{1}{\eps^{N-1}} \int_{\R^{N}} \frac{\xi_0}{\eps} 
\Big( R^2\Big( \frac{x-x_0}{\eps} \Big) - 
|v_\eps(x)|^2 \Big) + 
\frac{1}{\eps^N} \int_{\R^{N}} |v_\eps(x)|^2 \big( \xi_0 - \eps \nabla S_\eps(0,x) \big) \Big| \\
& \le |\xi_0|\, \frac{1}{\eps^N} \int_{\R^{N}} 
\Big| R^2\Big( \frac{x-x_0}{\eps} \Big) - |v_\eps(x)|^2 \Big| \\
& +  \Big( \frac{1}{\eps^N} \int_{\R^{N}} |v_\eps(x)|^2  \Big)^{\frac 12} 
\Big( \frac{1}{\eps^N} \int_{\R^{N}} |v_\eps(x)|^2 
\big| \xi_0 - \eps \nabla S_\eps(0,x) \big|^2 \Big)^{\frac 12} \\
& \le 2 |\xi_0| \sqrt{m \gamma} + \const(R,x_{0},\xi_{0}) 
\sqrt{m} \gamma^{\frac{1}{4}} \le  \const(R,x_{0},\xi_{0})\, \gamma^{\frac 14},
\end{align*}
where in the last line we have used the inequality for all $a,b\in \C$
$$
\int \big| |a|^{2}-|b|^{2} \big| \le 
\Big(\int (|a|+|b|)^{2} \Big)^{\frac 12} 
\Big(\int |a-b|^{2} \Big)^{\frac 12},
$$
condition (C2) on $v_{\eps}(x)$ and the estimate \eqref{stime-momento-1}.
The term $\eta_{2}^{\eps}(0)$ is estimated in Lemma \ref{stima-en-in-2}.
\end{proof}

\noindent
Let us now consider the increase rate in time. We can state the following

\begin{proposition}\label{viene-1}
For every $t\in [0,T^{\eps,\gamma})$, we have
$$
\Big|\frac{d}{dt}\eta_1^\eps(t) \Big| 
\le \const(V,R,x_{0},\xi_{0},v_{\eps})
\Big( |\eta^\eps(t)| + \sqrt{\gamma} 
+ \eps^2 + T\, \frac{(K_{\eps}(u_{\eps},t)-
mV_{0})^{\frac 14}}{\eps^{\frac 12 + 3 \frac{2+\beta}{1-\beta}}}  \Big),
$$
for every $\eps$ small enough.
\end{proposition}

\begin{proof}
Let $\theta^{\eps}$ be the family of functions introduced in Lemma \ref{rappr}. 
Then, using \eqref{newton} and \eqref{utile-2}, we have
\begin{align*}
\Big| \frac{d}{dt} \eta_1^\eps(t)\Big| & 
= \Big| m \nabla V(x(t)) - \frac{1}{\eps^N} \int_{\R^{N}} \nabla V(x) \, |u_\eps(t,x)|^2 \Big| \\
& = \Big|  \frac{1}{\eps^N} \int_{\R^{N}} |u_\eps(t,x)|^2 
\big[ \nabla V(x(t)) - \nabla V(x) \big] \Big| \le \\
& \le \Big| \frac{1}{\eps^N} \int_{\R^N} \Big( |u_\eps(t,x)| - R\Big( \frac{x-x^\eps(t)}{\eps}\Big) \Big)^2 \big[ \nabla V(x(t)) - \nabla V(x) \big] \Big| + \\
& + \Big| \frac{1}{\eps^N} \int_{\R^N} R^2\Big( \frac{x-x^\eps(t)}{\eps}\Big) \big[ \nabla V(x(t)) - \nabla V(x) \big] \Big| + \\
& + \Big| \frac{2}{\eps^N} \int_{\R^N} \Big( |u_\eps(t,x)| - R\Big( \frac{x-x^\eps(t)}{\eps}\Big) \Big) \, R\Big( \frac{x-x^\eps(t)}{\eps}\Big)\, \big[ \nabla V(x(t)) - \nabla V(x) \big] \Big| = \\
& =: I_{1} + I_{2} +I_3
\end{align*}
where we used the elementary identity $|a|^{2}= (|a|-|b|)^{2} + |b|^{2} +2 (|a|-|b|) |b|$.

\noindent Let us estimate these terms, beginning with $I_1$. 

$$I_1 = \Big| \frac{1}{\eps^N} \int_{\R^N} \Big( |u_\eps(t,x)| - R\Big( \frac{x-x^\eps(t)}{\eps}\Big) \Big)^2 \big[ \nabla V(x(t)) - \nabla V(x) \big] \Big| \le $$
$$ \le \frac{1}{\eps^N} \int_{\R^N} \Big| u_\eps(t,x) - R\Big( \frac{x-x^\eps(t)}{\eps}\Big) e^{\frac \imath \eps \, (\xi(t) \cdot x + \theta^\eps(t))} \Big|^2 \big| \nabla V(x(t)) - \nabla V(x) \big|$$

\noindent Let 
$\tilde \delta=\delta(x_{0},\xi_{0})/2$ so that $\text{supp}\, v_\eps 
\cap B(0,\tilde \delta) = \emptyset$ by assumption (C3), and introduce 
a cut-off function $\psi_{\tilde \delta}\in C^{\infty}_{0}(\R^{N})$ such that
\begin{equation} \label{la-mia-psi}
\psi_{\tilde \delta}(x) =
\begin{cases} 1, & |x|\le \frac{\tilde \delta}{2}, \\ 
	0, & |x|\ge \tilde \delta, 
	\end{cases}
	\qquad\,\,\, 
|\nabla \psi_{\tilde \delta}(x)| \le \frac{4}{\tilde \delta}, 
\quad\,\, \text{for}\ \ \frac{\tilde \delta}{2} \le |x| \le \tilde \delta.
\end{equation}
Then we can write
\begin{align*}
I_{1} &\le  \frac{2}{\eps^N}  \int_{\R^{N}} 
\Big| u_\eps(t,x) - R\Big( \frac{x-x^\eps(t)}{\eps}\Big) e^{\frac{\imath}{\eps}(\xi(t)\cdot x
+\theta^{\eps}(t))} \Big|^{2} | \nabla V(x(t)) - \nabla V(x) | 
\psi_{\tilde \delta}(x)   \\
&+ \frac{2}{\eps^N}  \int_{\R^{N}} \Big| u_\eps(t,x) - 
R\Big( \frac{x-x^\eps(t)}{\eps}\Big) e^{\frac{\imath}{\eps}(\xi(t)
\cdot x+\theta^{\eps}(t))} \Big|^{2}\, 
| \nabla V(x(t)) - \nabla V(x) | (1-\psi_{\tilde \delta}(x))   \\
& \le \frac{2}{\eps^N} \int_{B(0,\tilde \delta)} \Big| u_\eps(t,x) - 
R\Big( \frac{x-x^\eps(t)}{\eps}\Big) e^{\frac{\imath}{\eps}(\xi(t)\cdot x
+\theta^{\eps}(t))} \Big|^{2} 
(|\nabla V(x(t))| + |\nabla V(x)|) \psi_{\tilde \delta}(x)  \\
& + \frac{4\phi(\tilde \delta/2)}{\eps^{N}} \int_{\R^{N}\setminus B(0,
\tilde \delta/2)} \Big| u_\eps(t,x) - 
R\Big( \frac{x-x^\eps(t)}{\eps}\Big) e^{\frac{\imath}{\eps}(
\xi(t)\cdot x+\theta^{\eps}(t))} \Big|^{2},
\end{align*}
where $\phi$ is defined in~\eqref{il-phi}. By Lemma 
\ref{rappr}, inequality 
$|a-b|^{2}\le 2|a|^{2} + 2|b|^{2}$ and 
$\inf_{t\geq 0} |x(t)| \ge \delta$, we write
\begin{align*}
I_{1} &\le \frac{4}{\eps^N} \int_{B(0,\tilde \delta)} \left| u_\eps(t,x) 
\right|^{2}\, |\nabla V(x)|\, \psi_{\tilde \delta}(x)  
+ \frac{4}{\eps^N} \int_{B(0,\tilde \delta)} 
\Big| R\Big( \frac{x-x^\eps(t)}{\eps}\Big) \Big|^{2}\, |\nabla V(x)| \\
&+ \frac{2}{\eps^N}\, (\phi(\delta) + 2 \phi(\tilde \delta/2))\, 
\int_{\R^{N}} \Big| u_\eps(t,x) - R\Big( \frac{x-x^\eps(t)}{\eps}\Big) 
e^{\frac{\imath}{\eps}(\xi(t)\cdot x+ \theta^{\eps}(t))} \Big|^{2} \\
& \le \frac{4}{\eps^N} \int_{B(0,\tilde \delta)} \left| u_\eps(t,x) 
\right|^{2}\, |\nabla V(x)|\, \psi_{\tilde \delta}(x)  + 
\frac{4}{\eps^N} \int_{B(0,\tilde \delta)} \Big| R\Big( 
\frac{x-x^\eps(t)}{\eps}\Big) \Big|^{2}\, |\nabla V(x)|  \\
&+ 6 \max\{\phi(\delta), \phi(\tilde \delta/2)\}
{\mathcal C}(|\xi(t)|\, |\eta_1^\eps(t)| + |\eta_2^\eps(t)| 
+ \const(R,x_{0},\xi_{0})( \sqrt{\gamma} + \eps^2)) \\
&\le \frac{4}{\eps^N} \int_{B(0,\tilde \delta)} 
\left| u_\eps(t,x) \right|^{2}\, |\nabla V(x)|
\psi_{\tilde \delta}(x)  + 4\, \| \nabla V\|_{L^{1}}  
\frac{1}{\eps^{N}}\, \| R^{2} \|_{L^{\infty}(\R^{N}\setminus B(0,
\frac{|x^\eps(t)|-\tilde \delta}{\eps}))}  \\
&+ 6\, \phi(\delta/4)\, {\mathcal C}(|\xi(t)|\, 
|\eta_1^\eps(t)| + |\eta_2^\eps(t)| + \const(R,x_{0},\xi_{0})( \sqrt{\gamma} + \eps^2),
\end{align*}
and by \eqref{prop-R} it holds for $\eps$ small enough
$$
\| R^{2} \|_{L^{\infty}(\R^{N}\setminus B(0,\frac{|x^\eps(t)|-
\tilde \delta}{\eps}))} \le \const\, \frac{\eps^{N-1}
\, e^{-\frac{\delta}{2\eps}}}{\delta^{N-1}},
$$
where, since $x^\eps(t)=x(t)-\eps w^\eps(t)$, $|w^{\eps}(t)|\leq 1$
for $t\in [0,T^{\eps,\gamma})$ as defined in \eqref{t-eps-new} and $|x(t)|\geq \delta$, for $\eps$ small
$$
\frac{|x^\eps(t)|-\tilde \delta}{\eps}\geq
\frac{|x(t)|-\eps-\tilde \delta}{\eps} \ge 
\frac{\delta-\tilde \delta-\eps}{\eps} \geq  \frac{\delta}{4\eps}.
$$ 
Hence we choose $\eps_{0}>0$ such that, for $\eps< \eps_0$ 
\begin{equation}\label{eps-0}
\frac{1}{\eps^{N}} 
\frac{\eps^{N-1}
e^{-\frac{\delta}{2\eps}}}{\delta^{N-1}} < \eps^{2}.
\end{equation}
We obtain then
\begin{equation}\label{stima-i1}
I_{1}\le \frac{4}{\eps^N} 
\int_{B(0,\tilde \delta)} \left| u_\eps(t,x) \right|^{2}\, |\nabla V(x)|\, \psi_{\tilde \delta}(x)  
+\const(V,R,x_{0},\xi_{0})\, {\mathcal C}(|\xi(t)|\, 
|\eta_1^\eps(t)| + |\eta_2^\eps(t)| + \sqrt{\gamma} + \eps^2).
\end{equation}
We conclude the proof by showing that, for every 
$t\in [0,T^{\eps,\gamma})$, there holds
\begin{equation}
\label{ex-lemma}
\frac{1}{\eps^N} \int_{B(0,\tilde \delta)} \left| u_\eps(t,x) \right|^{2}
|\nabla V(x)|\, \psi_{\tilde \delta}(x)  \le \const(V,x_{0},\xi_{0},v_{\eps}) 
\Big( \eps^{2} + T\, \frac{(K_{\eps}(u_{\eps},t)-mV_{0})^{\frac 14}}{\eps^{\frac 12 + 3 \frac{2+\beta}{1-\beta}}}\Big).
\end{equation}
Let us introduce another cut-off at the 
	origin, that is a function $\varphi_{\eps} 
	\in C^{\infty}_{0}(\R^{N}\setminus \set{0})$, satisfying
	\begin{equation}\label{la-mia-fi}
	\varphi_{\eps}(x) =
	\left\{ \begin{array}{ll} 0, & |x|\le r''_{\eps}, 
	\\[0.2cm] 1, & r'_{\eps} \le |x| \le 2\tilde \delta, \\[0.2cm] 0, & |x|\ge 3 \tilde \delta, 
	\end{array} \right. \qquad |\nabla \varphi_{\eps}(x)| 
	\le \frac{2}{r'_{\eps} - r''_{\eps}}, \quad \text{for}\ \ r''_{\eps}\le |x| \le r'_{\eps},
	\end{equation}
	with $r'_{\eps}$ and $r''_{\eps}$ to be chosen later, 
	see formulas \eqref{scelta-raggi}. By assumption (V1) and 
	inequality \eqref{gagl-nir} with the choice $q= 2^{*}$, 
	we apply H\"older inequality to obtain
	\begin{align*}
	& \frac{1}{\eps^N} \int_{B(0,\tilde \delta)} \left| u_\eps(t,x) \right|^{2}\, |\nabla V(x)|\, \psi_{\tilde \delta}(x) (1-\varphi_{\eps}(x)) 
	\le \frac{1}{\eps^N} \int_{B(0,r'_{\eps})} \left| u_\eps(t,x) \right|^{2}\, |\nabla V(x)|  \\
	&\le \const(N)\, \frac{1}{\eps^N} \|\nabla u_{\eps}(t,\cdot)\|_{L^{2}}^{2} 
	\left( \int_{0}^{r'_{\eps}}\, \frac{1}{r^{\frac N2(\beta +1)}}\, r^{N-1}\, 
	dr \right)^{\frac 2N} \le \const(N,\beta) \, M(x_{0},\xi_{0}, v_{\eps})\, \frac{(r'_{\eps})^{1-\beta}}{\eps^{2}},
\end{align*}
	where $M(x_{0},\xi_{0}, v_{\eps})$ is defined in Lemma \ref{stime-h1}. We can then write
	\begin{equation}\label{mcp}
	\frac{1}{\eps^N} \int_{B(0,\tilde \delta)} 
	\left| u_\eps(t,x) \right|^{2}\, |\nabla V(x)|\, \psi_{\tilde \delta}(x)  
	\leq \frac{1}{\eps^N} 
	\int_{B(0,\tilde \delta)} \left| u_\eps(t,x) \right|^{2}\, |\nabla V(x)|\, \psi_{\tilde \delta}(x) \varphi_{\eps}(x) + 
	\const\, \frac{(r'_{\eps})^{1-\beta}}{\eps^{2}},
	\end{equation}
	where the constant in the last term only depends on the initial conditions of \eqref{ivp}.
	Moreover, by definition of $\tilde \delta$ and by virtue of identity \eqref{utile-1}, we have
	$$
	\int_{B(0,\tilde \delta)} 
	\left| u_\eps(0,x) \right|^{2}\, |\nabla V(x)|\, 
	\psi_{\tilde \delta}(x) \varphi_{\eps}(x) = 0
	 $$	
Since $\psi_{\tilde \delta}(x) \varphi_{\eps}(x) 
	\in C^{\infty}_{0}(B(0,\tilde \delta)\setminus \set{0})$, there holds
	\begin{equation*}
	\frac{d}{dt} \Big( \int_{B(0,\tilde \delta)} 
	\frac{|u_\eps(t,x)|^{2}}{\eps^N}\, |\nabla V(x)|
	\, \psi_{\tilde \delta}(x) \varphi_{\eps}(x) \Big) 
	= \int_{B(0,\tilde \delta)} p_{\eps}(t,x) \cdot \nabla 
	\left( |\nabla V(x)|\, \psi_{\tilde \delta}(x) \varphi_{\eps}(x) \right),
\end{equation*}
	and to give an estimate for this last term we use the radial 
	notation \eqref{p-radiale} for the momentum 
	density and split the integral in three terms, where the 
	properties of the cut-off functions $\psi_{\tilde \delta}$, 
	see \eqref{la-mia-psi}, and $\varphi_{\eps}$, 
	see \eqref{la-mia-fi}, are used to determine the domain of integration. We obtain
	\begin{equation}\label{vanno-qui}
	\Big| \int_{B(0,\tilde \delta)} p_{\eps}(t,x) \cdot \nabla 
	\left( |\nabla V(x)|\, \psi_{\tilde \delta}(x) 
	\varphi_{\eps}(x) \right) \Big| \le J_{1} + J_{2} + J_{3},
	\end{equation}
	where we have set
	\begin{align*}
	J_{1} &:= \frac{1}{\eps^{N}} \int_{B(0,\tilde \delta) 
	\setminus B(0,r''_{\eps})}\, |u_{\eps}(t,x)|^{2} \eps 
	\, \left| \nabla S_{\eps}(t,x) \right|\,  | \nabla |\nabla V(x)| |, \\
	J_{2} &:= \frac{1}{\eps^{N}} \int_{B(0,\tilde \delta) 
	\setminus B(0,\tilde \delta/2)}\, 
	|u_{\eps}(t,x)|^{2} \eps \, \left| \nabla 
	S_{\eps}(t,x) \right|\,  |\nabla V(x)|\, |\nabla \psi_{\tilde \delta}(x) |, \\
	J_{3} &:= \frac{1}{\eps^{N}} \int_{B(0,r'_{\eps}) 
	\setminus B(0,r''_{\eps})}\, |u_{\eps}(t,x)|^{2} \eps  
	\left| \nabla S_{\eps}(t,x) \right|\,  |\nabla V(x)|\, |\nabla \varphi_{\eps}(x) | .
\end{align*}
	The estimates for the $J_i$s are 
	similar. We use H\"older inequality, 
	assumptions (V1)-(V3) and the estimate 
	\eqref{stima-kin} for the kinetic energy 
	$K_{\eps}(u_{\eps},t)$ defined in \eqref{kin-energy}. We obtain
	$$
	\eps^{N} J_{1}\le \|u_{\eps}(t,\cdot)\|_{L^{2^{*}}}^{\frac 12}
	\Big( \int |u_{\eps}(t,x)|^{2}\, \eps^{2} |\nabla S_{\eps}(t,x)|^{2} \Big)^{\frac 12}\, 
	\Big( \int |u_{\eps}(t,x)|^{2}\, (V(x)-V_{0}) \Big)^{\frac 14}  
	\Big( \int \frac{\big| \nabla |\nabla V(x)| \big|^{2N}}{(V(x)-V_{0})^{\frac N2}} \Big)^{\frac{1}{2N}},
	$$
	where all the integrals are computed on the set 
	$B(0,\tilde \delta) \setminus B(0,r''_{\eps})$. Moreover we have 
	$$
	\|u_{\eps}(t,\cdot)\|_{L^{2^{*}}}^{\frac 12} \le \const(N) 
	\| \nabla u_{\eps}(t,\cdot)\|_{L^{2}}^{\frac 12} \le \const(N,x_{0},\xi_{0},v_{\eps})\, \eps^{\frac{N-2}{4}},
	$$
	by the Gagliardo-Nirenberg inequality \eqref{gagl-nir} and Lemma \ref{stime-h1},
	$$
	\Big( \int_{B(0,\tilde \delta) \setminus B(0,r''_{\eps})} |u_{\eps}(t,x)|^{2}\, \eps^{2} 
	|\nabla S_{\eps}(t,x)|^{2} \Big)^{\frac 12} \le 
	\Big( \int_{\R^{N}} |u_{\eps}(t,x)|^{2}\, \eps^{2} 
	|\nabla S_{\eps}(t,x)|^{2} \Big)^{\frac 12} \le \big(2\eps^{N} K_{\eps}(u_{\eps},t) \big)^{\frac 12},
	$$
	by definition of $K_{\eps}(u_{\eps},t)$ and the non-negativity of $V$,
	$$
	\Big( \int_{B(0,\tilde \delta) \setminus B(0,r''_{\eps})} |u_{\eps}(t,x)|^{2}\, (V(x)-V_{0})  \Big)^{\frac 14} \le \left( \int_{\R^{N}} 
	|u_{\eps}(t,x)|^{2}\, (V(x)-V_{0}) \right)^{\frac 14} 
	\le \big( \eps^{N}\, (K_{\eps}(u_{\eps},t) - m\, V_{0}) \big)^{\frac 14},
	$$
	by definition of $K_{\eps}(u_{\eps},t)$ and the conservation of mass,
	$$
	\Big( \int_{B(0,\tilde \delta) \setminus B(0,r''_{\eps})} 
	\frac{| \nabla |\nabla V(x)| |^{2N}}{(V(x)-V_{0})^{\frac N2}} \Big)^{\frac{1}{2N}} 
	\le \const(N)\, \Big( \int_{r''_{\eps}}^{\tilde \delta} 
	\frac{r^{-(\beta+2)2N}}{r^{-\beta\, \frac N2}} r^{N-1} dr \Big)^{\frac{1}{2N}} 
	= \frac{\const(\delta,\beta,N)}
	{\left( r''_{\eps} \right)^{\frac 34 (2+\beta)}},
	$$
	by assumptions on the behavior of $V$ around the origin. Hence, putting the above facts together, we get
	\begin{equation}\label{stima-j1}
	J_{1}\le \const(V,x_{0},\xi_{0},v_{\eps})
	(K_{\eps}(u_{\eps},t) )^{\frac 12} 
	\frac{(K_{\eps}(u_{\eps},t) - m V_{0})^{\frac 14}}{\left( r''_{\eps} \right)^{
	\frac 34 (2+\beta)}\, \eps^{\frac 12}}.
	\end{equation}
	For the term $J_{2}$, we write
	$$
	\eps^{N}\, J_{2} \le\frac{4}{\tilde \delta}\, 
	\Big( \int |u_{\eps}(t,x)|^{2}\, \eps^{2} |\nabla S_{\eps}(t,x)|^{2} 
	\Big)^{\frac 12} \Big( \int |u_{\eps}(t,x)|^{2}\, (V(x)-V_{0}) 
	\Big)^{\frac 12} \Big\| \frac{|\nabla V(x)|}{\sqrt{V(x)-V_{0}}} 
	\Big\|_{L^{\infty}(B(0,\tilde \delta) \setminus B(0,\tilde \delta/2))}
	$$
	where all the integrals are 
	computed on $B(0,\tilde \delta) \setminus B(0,\tilde \delta/2)$. 
	For the first two integrals we proceed just as above. Concerning
	the third term, on account of conditions (V1) and (V3), we have
	$$
	\Big\| \frac{|\nabla V(x)|}{\sqrt{V(x)-V_{0}}} \Big\|_{L^{\infty}
	(B(0,\tilde \delta) \setminus B(0,\tilde \delta/2))} 
	\le \frac{\phi(\tilde \delta/2)}{\tilde \delta^{\frac \beta 2}}.
	$$
	Hence, in turn, we can conclude
	\begin{equation}\label{stima-j2}
	J_{2}\le \const(V,x_{0},\xi_{0},v_{\eps})
	(K_{\eps}(u_{\eps},t) )^{\frac 12}\, (K_{\eps}(u_{\eps},t) - m\, V_{0})^{\frac 12}.
	\end{equation}
	Finally, concerning the term $J_{3}$, we write
	$$
	\eps^{N} J_{3}
	\le \frac{2 \|u_{\eps}(t,\cdot)\|_{L^{2^{*}}}^{\frac 12}}{r'_{\eps}-r''_{\eps}}
	\Big( \int |u_{\eps}(t,x)|^{2}\, \eps^{2} 
	|\nabla S_{\eps}(t,x)|^{2} \Big)^{\frac 12} 
	\Big( \int |u_{\eps}(t,x)|^{2} (V(x)-V_{0}) 
	\Big)^{\frac 14}\,  
	\Big( \int \frac{|\nabla V(x)|^{2N}}{(V(x)-V_{0})^{\frac N2}} \Big)^{\frac{1}{2N}},
	$$
	where all the integrals are computed on the set $B(0,r'_{\eps}) \setminus B(0,r''_{\eps})$. 
	For the first three terms above we proceed as for $J_{1}$. Concerning the last term, we write
	\begin{align*}
	\Big( \int_{B(0,r'_{\eps}) 
	\setminus B(0,r''_{\eps})} 
	\frac{|\nabla V(x)|^{2N}}{(V(x)-V_{0})^{\frac N2}} 
	\Big)^{\frac{1}{2N}} &\le \const(N)\, 
	\Big( \int_{r''_{\eps}}^{r'_{\eps}} \frac{r^{-(\beta+1)2N}}{r^{-\beta\, \frac N2}}\, 
	r^{N-1}\, dr \Big)^{\frac{1}{2N}} \\ 
	&= \const(\delta,\beta,N) \big( \big( r''_{\eps} \big)^{-N(1 +\frac 32\beta)} 
	- \big( r'_{\eps} \big)^{-N(1 +\frac 32\beta)} \big)^{\frac{1}{2N}} .
\end{align*}
	Hence we finally get
	\begin{equation}\label{stima-j3}
	J_{3}\le \const(V,x_{0},\xi_{0},v_{\eps}) (K_{\eps}(u_{\eps},t) )^{\frac 12}\, 
	(K_{\eps}(u_{\eps},t) - m\, V_{0})^{\frac 14} \frac{\big( \left( r''_{\eps} \right)^{-N(1 +\frac 32\beta)} - \left( r'_{\eps} \right)^{-N(1 +\frac 32\beta)} \big)^{\frac{1}{2N}}}{\eps^{\frac 12}(r'_{\eps}-r''_{\eps})}
	\end{equation}
	The proof of the inequality~\eqref{ex-lemma} is finished by choosing
	\begin{equation}\label{scelta-raggi}
	r'_{\eps} = \eps^{\frac{4}{1-\beta}},\qquad r''_{\eps} = \frac 12\, r'_{\eps}
	\end{equation}
	taking~\eqref{stima-kin} into account and 
	using \eqref{mcp} and \eqref{vanno-qui} together with \eqref{stima-j1}, \eqref{stima-j2} and \eqref{stima-j3}.
This concludes the proof of the estimate of $I_1$.

Concerning the second term $I_{2}$, at first, take $\tilde{\delta}$ as above. 
Choosing $\varepsilon$ sufficiently small, as in (\ref{eps-0}), and in reasoning in a similar way, we have 
$$
\frac{2}{\varepsilon^{N}}\int_{B(0,\tilde{\delta})}R^{2}
\left(\frac{x-x^{\varepsilon}(t)}{\varepsilon}\right)|\nabla V(x(t))-\nabla V(x)|\le \text{const}(V,R)\, \varepsilon^{2}.
$$
So, we consider $\tilde{V}\in C^{2}(\mathbb{R}^{N},\mathbb{R})$ such
that $\tilde{V}(x)=V(x)$ on $\mathbb{R}^{N}\smallsetminus B(0,\tilde{\delta}$)
and we get
$$
I_{2}\le\left|\frac{2}{\varepsilon^{N}}\int_{\mathbb{R}^{N}}R^{2}
\left(\frac{x-x^{\varepsilon}(t)}{\varepsilon}\right)[\nabla \tilde{V}(x(t))-\nabla\tilde{V}(x)]\right|+
\text{const}(V,R)\varepsilon^{2}.
$$
It holds
\begin{gather*}
\left|\int_{\mathbb{R}^{N}}R^{2}
\left(\frac{x-x^{\varepsilon}(t)}{\varepsilon}\right)[\nabla \tilde{V}(x(t))-\nabla\tilde{V}(x)]\right|\le \\
\left|\int_{\mathbb{R}^{N}}R^{2}
\left(\frac{x-x^{\varepsilon}(t)}{\varepsilon}\right)[\nabla \tilde{V}(x^\varepsilon(t))-\nabla\tilde{V}(x)]\right|+
\left|\int_{\mathbb{R}^{N}}R^{2}
\left(\frac{x-x^{\varepsilon}(t)}{\varepsilon}\right)[\nabla \tilde{V}(x(t))-\nabla\tilde{V}(x^\eps(t))]\right|
\end{gather*}
And, in light of Lemma \ref{w-come-in-k}, it holds
$$
\left|\frac 2{\eps^N}\int_{\mathbb{R}^{N}}R^{2}
\left(\frac{x-x^{\varepsilon}(t)}{\varepsilon}\right)[\nabla \tilde{V}(x(t))-\nabla\tilde{V}(x^\eps(t))]\right|\le
\text{const}(R,x_0,\xi_0,\delta,V)(|\eta^{\eps}(t)| +\sqrt{\gamma} + \eps^{2}).
$$
So we have that
$$
I_{2}\le
\left|
\frac{2}{\varepsilon^{N}}\int_{\mathbb{R}^{N}}R^{2}\left(\frac{x-x^{\varepsilon}(t)}{\varepsilon}\right)
[\nabla \tilde{V}(x^{\varepsilon}(t))-\nabla\tilde{V}(x)] dx
\right|+
\text{const}(R,x_0,\xi_0,\delta,V)(|\eta^{\eps}(t)| +\sqrt{\gamma} + \eps^{2}).
$$
Let us first write 
$$
\frac{2}{\varepsilon^{N}}\int_{\mathbb{R}^{N}}R^{2}\left(\frac{x-x^{\varepsilon}(t)}{\varepsilon}\right)
[\nabla \tilde{V}(x^{\varepsilon}(t))-\nabla\tilde{V}(x)]dx 
=2\int_{\mathbb{R}^{N}}R^{2}\left(y\right)
[\nabla \tilde{V}(x^{\varepsilon}(t))
-\nabla\tilde{V}(x^{\varepsilon}(t)+\varepsilon y)]dy.
$$
By virtue of \cite[Lemma 3.3]{keraa} we conclude 
$$
I_{2}\le \text{const}(R,x_0,\xi_0,\delta,V)(|\eta^{\eps}(t)| +\sqrt{\gamma} + \eps^{2}).
$$
For $I_3$ we write
$$I_3 \le 2 \left( \int_{\R^N}\, \frac{1}{\eps^N}\, 
\Big| u_\eps(t,x) - R\Big( \frac{x-x^\eps(t)}{\eps}\Big) \Big|^2 \right)^{\frac 12} \, \left( \int_{\R^N} \,  \frac{1}{\eps^N} 
\big| \nabla V(x(t)) - \nabla V(x) \big|^2 R^2\Big( \frac{x-x^\eps(t)}{\eps}\Big) \right)^{\frac 12}\le$$
$$\le \const(R,x_0,\xi_0)\, \left( |\eta_1^\eps(t)| + |\eta_2^\eps(t)| + \sqrt{\gamma} + \eps^2 \right)$$
arguing as in the previous estimates.

\noindent This concludes the proof.
\end{proof}

\noindent
For $\eps$ small, let us set
\begin{equation}
	\label{sempre-raggi}
\rho'_\eps = \eps^\frac{4}{2-\beta}\, , \qquad \rho''_\eps = \frac 12 \, \rho'_\eps,
\end{equation}
introduce a cut-off function
\begin{equation}\label{il-cut-off}
\chi_{\eps}(x) =
\begin{cases} 
	0 & |x|\le \rho''_{\eps}, \\ 
	1 & |x| \ge \rho'_{\eps}, 
\end{cases}
 \qquad 
|\nabla \chi_{\eps}(x)| \le \frac{2}{\rho'_{\eps} - \rho''_{\eps}} 
\quad\,\, \text{for $\rho''_{\eps}\le |x| \le \rho'_{\eps}$},
\end{equation}
and, finally, define
$$
\tilde \eta_{2}^{\eps}(t) := m V(x(t)) - \frac{1}{\eps^{N}} 
\int_{\R^{N}} |u_{\eps}(t,x)|^{2}\, V(x)\, \chi_{\eps}(x),\quad t\in [0,T^{\eps,\gamma}).
$$
Then, we have the following

\begin{proposition}\label{viene-2}
For every $t\in [0,T^{\eps,\gamma})$ we have
$|\eta_{2}^{\eps}(t)| \le |\tilde \eta_{2}^{\eps}(t)| + \const(x_{0},\xi_{0},v_{\eps}) \eps^2$, with
$$|\tilde \eta_{2}^{\eps}(0)| \le \const(R,x_{0},\xi_{0},v_{\eps})\, (\gamma + \eps^{2})$$ 
and
$$
\Big|\frac{d}{dt}\tilde \eta_2^\eps(t) \Big| 
\le \const(V,R,x_{0},\xi_{0},v_{\eps})
\Big[ |\eta^\eps(t)| 
+ \sqrt{\gamma} + \eps^2 + (K_{\eps}(u_{\eps},t)-mV_{0})^{\frac 14} \, \Big( \frac{1}{\eps^{1 + 2 \frac{2+3\beta}{2-\beta}}} + \frac{T}{\eps^{\frac 12 + 3\, \frac{2+\beta}{1-\beta}}} \Big) \Big].
$$
\end{proposition}

\begin{proof}
We first estimate the behavior of $\eta_{2}^{\eps}$ near the origin. We can write
$$
\eta_{2}^{\eps}(t) = m V(x(t)) - \frac{1}{\eps^{N}} 
\int_{\R^{N}} |u_{\eps}(t,x)|^{2} V(x) 
\chi_{\eps}(x)  - \frac{1}{\eps^{N}}\, \int_{\R^{N}} |u_{\eps}(t,x)|^{2}\, V(x)\, (1-\chi_{\eps}(x)) 
$$
Moreover by H\"older inequality, inequality \eqref{gagl-nir}, 
Lemma \ref{stime-h1} and assumption (V1), 
\begin{align*}
\Big| \frac{1}{\eps^{N}} \int_{\R^{N}} |u_{\eps}(t,x)|^{2} V(x)\, (1-\chi_{\eps}(x)) \Big| 
&\le \frac{1}{\eps^{N}}\, \int_{B(0,\rho'_{\eps})} \, |u_{\eps}(t,x)|^{2}\, V(x) \\
&\le \frac{1}{\eps^{N}}\, \| u_{\eps} \|_{L^{2^{*}}}^{2}
\Big( \int_{B(0,\rho'_{\eps})} V(x)^{\frac N2} \Big)^{\frac 2N} \le 
\const(x_{0},\xi_{0},v_{\eps})\, \frac{(\rho'_{\eps})^{2-\beta}}{\eps^{2}}.
\end{align*}
Whence, there holds
$$
|\eta_{2}^{\eps}(t)| \le |\tilde \eta_{2}^{\eps}(t)| 
+ \const(x_{0},\xi_{0},v_{\eps})\, \frac{(\rho'_{\eps})^{2-\beta}}{\eps^{2}} = |\tilde \eta_{2}^{\eps}(t)| 
+ \const(x_{0},\xi_{0},v_{\eps})\, \eps^{2}.
$$
by \eqref{sempre-raggi}. Using also Lemma \ref{eta-iniz} the estimate for $|\tilde \eta_{2}^{\eps}(0)|$ follows.

Using formulas~\eqref{newton} and \eqref{utile-1} and the radial notation 
\eqref{p-radiale} for the momentum density, we have
\begin{align*}
\Big| \frac{d}{dt}\tilde \eta_2^\eps(t)\Big| &= \Big| m \nabla V(x(t)) \cdot \xi(t) 
+ \int_{\R^N} (\nabla\cdot p_{\eps}(t,x)) V(x) \chi_\eps(x) \Big| \\
& \le \Big| \frac{1}{\eps^{N}}  \int_{\R^N} 
|u_{\eps}(t,x)|^{2}\, (\nabla V(x(t))\cdot \xi(t)) 
\chi_\eps(x) - \frac{1}{\eps^N} 
\int_{\R^N} |u_\eps(t,x)|^2  (\eps \nabla S_{\eps}(t,x) \cdot \nabla V(x) )\, \chi_\eps(x)  \Big| \\
&+ \Big| \frac{1}{\eps^{N}}  \int_{\R^N} |u_{\eps}(t,x)|^{2}\, 
(\nabla V(x(t))\cdot \xi(t))\, (1-\chi_\eps(x)) \Big| \\
&+ \Big| \frac{1}{\eps^N} \int_{\R^N} |u_\eps(t,x)|^2 
(\eps \nabla S_{\eps}(t,x) \cdot \nabla \chi_\eps(x) ) V(x)  \Big| 
=: I_1 + I_2 +I_3.
\end{align*}
Let us estimate these terms, beginning with $I_1$. By adding and subtracting 
$|u_\eps(t,x)|^2 \nabla V(x) \cdot \xi(t)$, we write
\begin{align*}
I_1 &\le \Big| \frac{1}{\eps^{N}} \, \int_{\R^N} |u_{\eps}(t,x)|^{2} 
\big[ \nabla V(x(t)) - \nabla V(x) \big]\cdot \xi(t)\, \chi_\eps(x)  \Big|  \\
& + \Big| \frac{1}{\eps^{N}} \, \int_{\R^N} |u_{\eps}(t,x)|^{2} 
\nabla V(x) \cdot \big[ \xi(t) - \eps \nabla S_{\eps}(t,x) \big]  \chi_\eps(x) \Big| \\
& \le \const(V,R,x_{0},\xi_{0},v_{\eps})
\Big( |\eta^\eps(t)| + \sqrt{\gamma} 
+ \eps^2 + T\, \frac{(K_{\eps}(u_{\eps},t)-
mV_{0})^{\frac 14}}{\eps^{\frac 12 + 3 \frac{2+\beta}{1-\beta}}}  \Big) + \\
& + \Big| \frac{1}{\eps^{N}} 
\, \int_{\R^N} |u_{\eps}(t,x)|^{2}\, \nabla V(x) 
\cdot \big[ \xi(t) - \eps \nabla S_{\eps}(t,x) \big]  \chi_\eps(x) \Big|,
\end{align*}
where we have used the estimate of Proposition \ref{viene-1} for the derivative of $\eta_1^\eps$. Moreover
\begin{align*}
& \Big| \frac{1}{\eps^{N}}  \int_{\R^N} |u_{\eps}(t,x)|^{2} 
\nabla V(x) \cdot \big[ \xi(t) - \eps \nabla S_{\eps}(t,x) \big]  \chi_\eps(x) \Big| \\
&\le \frac{1}{2\eps^{N}} \, \int_{\R^N} |u_{\eps}(t,x)|^{2} \big| \xi(t) - \eps \nabla 
S_{\eps}(t,x) \big|^2 + \frac{1}{2\eps^{N}} \, \int_{\R^N} |u_{\eps}(t,x)|^{2}\, |\nabla V(x)|^2 \chi_\eps(x) \\
&\le \frac{1}{2}|\xi(t)| |\eta_1^\eps(t)| + \frac{1}{2}|\eta_2^\eps(t)| + \const(R,x_0,\xi_0) (\sqrt{\gamma} + \eps^2) 
+ \frac{1}{2\eps^{N}}  \int_{\R^N} |u_{\eps}(t,x)|^{2}\, |\nabla V(x)|^2 \chi_\eps(x),
\end{align*}
by virtue of inequality~\eqref{stime-momento-2}. Finally, by H\"older inequality and the definition of 
$\chi_\eps$ in \eqref{il-cut-off}, we get
$$
\int_{\R^N} |u_{\eps}(t,x)|^{2}\, |\nabla V(x)|^2 \chi_\eps(x)  
\le \|u_\eps(t,\cdot)\|_{L^{2^*}}\, \Big( \int_{\R^N}\, |u_\eps(t,x)|^2 
(V(x)-V_0) \Big)^{\frac 12} \Big( 
\int_{\R^N}\, \frac{|\nabla V(x)|^{2N}}{(V(x)-V_0)^{\frac N2}} \chi_\eps^N(x)  \Big)^{\frac 1N},
$$
and we can use the estimates
$$
\|u_{\eps}(t,\cdot)\|_{L^{2^{*}}} \le \const(N) 
\| \nabla u_{\eps}(t\cdot)\|_{L^{2}} \le \const(N,x_{0},\xi_{0},v_{\eps})\, \eps^{\frac{N-2}{2}},
$$
via inequality \eqref{gagl-nir} and Lemma \ref{stime-h1},
$$
\Big( \int_{\R^N}\, |u_\eps(t,x)|^2\, (V(x)-V_0) 
\Big)^{\frac 12} \le ( \eps^{N} (K_{\eps}(u_{\eps},t) - m V_{0}) )^{\frac 12},
$$
by definition of $K_{\eps}(u_{\eps},t)$ and the conservation of mass,
$$
\Big( \int_{\R^N} \frac{|\nabla V(x)|^{2N}}{(V(x)-V_0)^{\frac N2}} 
\chi_\eps^N(x)  \Big)^{\frac 1N} \le \const(V,N)\, \Big( \int_{\rho''_{\eps}}^{1} 
\frac{r^{-(\beta+1)2N}}{r^{-\beta\, \frac N2}}\, 
r^{N-1}\, dr \Big)^{\frac{1}{N}} 
= \const(V,\delta,\beta)\, ( \rho''_{\eps})^{-(1 +\frac 32\beta)},
$$
by assumptions (V1) and (V2). Hence, we obtain
\begin{equation}\label{stima-i1-new}
I_1\le \const(R,V,x_0,\xi_0) \Big[ |\eta_1^\eps(t)| + |\eta_2^\eps(t)| + \sqrt{\gamma} + \eps^2 + T\, \frac{(K_{\eps}(u_{\eps},t)-
mV_{0})^{\frac 14}}{\eps^{\frac 12 + 3 \frac{2+\beta}{1-\beta}}} + \frac{(K_{\eps}(u_{\eps},t) 
- m\, V_{0})^{\frac 12}}{\eps ( \rho''_{\eps} )^{(1 +\frac 32\beta)}} \Big].
\end{equation}
We now turn to the estimate for $I_2$. By assumption (V3),
\eqref{deltabblow} and the definition of $\chi_\eps$, we write
$$
I_2 \le \phi(\delta) |\xi(t)| \frac{1}{\eps^{N}}  \int_{B(0,\rho'_\eps)} |u_{\eps}(t,x)|^{2},
$$
and by H\"older inequality, \eqref{gagl-nir} and Lemma \ref{stime-h1}, we obtain
\begin{equation}\label{stima-i2-new}
I_2\le \const(R,V,x_0,\xi_0)\, |\xi(t)|\, \eps^{-N}\, \|u_\eps\|_{L^{2^*}}^2\, (\rho'_\eps)^2 \le 
\const(R,V,x_0,\xi_0,v_\eps)\, |\xi(t)|\, \frac{(\rho'_\eps)^2}{\eps^2}
\end{equation}
We now estimate $I_3$. We apply again H\"older inequality and the properties of $\chi_\eps$ to get
$$
\eps^N\, I_3\le\frac{2\, \|u_\eps\|_{L^{2^*}}^{\frac 12}}{\rho'_\eps - \rho''_\eps}\, 
\Big( \int |u_{\eps}(t,x)|^{2}\, \eps^{2} |\nabla S_{\eps}(t,x)|^{2}\Big)^{\frac 12}\, 
\Big( \int |u_{\eps}(t,x)|^{2}\, (V(x)-V_{0}) \Big)^{\frac 14}\,  
\Big( \int \frac{V(x)^{2N}}{(V(x)-V_{0})^{\frac N2}} \Big)^{\frac{1}{2N}},
$$
where all integrals are computed on 
$B(0,\rho'_\eps) \setminus B(0,\rho''_\eps)$. Hence, we have
$$
\|u_{\eps}(t,\cdot)\|_{L^{2^{*}}}^{\frac 12} \le \const(N)\, 
\| \nabla u_{\eps}(t,\cdot)\|_{L^{2}}^{\frac 12} \le \const(N,x_{0},\xi_{0},v_{\eps})\, \eps^{\frac{N-2}{4}}
$$
by inequality \eqref{gagl-nir} and Lemma \ref{stime-h1},
$$
\Big( \int_{B(0,\rho'_\eps) \setminus B(0,\rho''_\eps)} 
|u_{\eps}(t,x)|^{2}\, \eps^{2} |\nabla S_{\eps}(t,x)|^{2} 
\Big)^{\frac 12} \le 
\Big( \int_{\R^{N}} |u_{\eps}(t,x)|^{2}\, \eps^{2} |\nabla S_{\eps}(t,x)|^{2} 
\Big)^{\frac 12} \le \big(\eps^{N}\, K_{\eps}(u_{\eps},t) \big)^{\frac 12}
$$
by definition of $K_{\eps}(u_{\eps},t)$ and by the non-negativity of $V$,
$$
\Big( \int_{B(0,\rho'_\eps) \setminus B(0,\rho''_\eps)}\, 
|u_\eps(t,x)|^2\, (V(x)-V_0) \Big)^{\frac 14} \le 
\Big( \int_{\R^N}\, |u_\eps(t,x)|^2\, (V(x)-V_0) \Big)^{\frac 14} 
\le \big( \eps^{N} (K_{\eps}(u_{\eps},t) - m V_{0}) \big)^{\frac 14},
$$
by definition of $K_{\eps}(u_{\eps},t)$ and the conservation of mass,
\begin{align*}
\Big( \int_{B(0,\rho'_\eps) \setminus B(0,\rho''_\eps)}\, 
\frac{V(x)^{2N}}{(V(x)-V_0)^{\frac N2}} \Big)^{\frac 1N} 
&\le \const(V,N)\, \Big( \int_{\rho''_{\eps}}^{\rho'_\eps} 
r^{-\beta\, \frac 32 N} \, r^{N-1}\, dr \Big)^{\frac{1}{N}} \\
&= \const(\delta,\beta,N)\, \big| 
\left( \rho''_{\eps} \right)^{N(1 -\frac 32\beta)} 
- \left( \rho'_{\eps} \right)^{N(1 -\frac 32\beta)} \big|^{\frac{1}{N}} ,
\end{align*}
by the assumptions (V1). Hence
\begin{equation}\label{stima-i3-new}
I_3\le \const(R,V,x_0,\xi_0)\, 
(K_{\eps}(u_{\eps},t))^{\frac 12}\, 
(K_{\eps}(u_{\eps},t) - m\, V_{0})^{\frac 14} \frac{\big| \left( \rho''_{\eps} \right)^{N(1 -\frac 32\beta)} - 
\left( \rho'_{\eps} \right)^{N(1 -\frac 32\beta)} 
\big|^{\frac{1}{N}}}{\eps^{\frac 12}\, (\rho'_\eps-\rho''_\eps)}.
\end{equation}
Taking into account~\eqref{sempre-raggi} the assertion finally follows from 
inequalities \eqref{stima-i1-new}, \eqref{stima-i2-new} and \eqref{stima-i3-new}.
\end{proof}

\section{Proof of the main result completed} \label{finale}

\noindent
Taking into account conditions~\eqref{cond-small} and inequality~\eqref{stima-kin}, we can find a 
$\const(R,\xi_0)$ such that
$$
K_{\eps}(u_{\eps},t) - m\, V_{0} \le \frac 12 m |\xi_{0}|^{2} + \frac{1}{\eps^{N}} \, \int_{\R^{N}} (V(x)-V_0) |v_{\eps}(x)|^{2} + 
\const(R,\xi_{0})\, \sqrt{\gamma} \le \eps^{2\frac{17+\beta}{1-\beta}}.
$$
Then, by Propositions \ref{viene-1}-\ref{viene-2}, for all $t\in [0,T^{\eps,\gamma})$ 
we have $|\eta_{2}^{\eps}(t)| \le \const(x_{0},\xi_{0},v_{\eps}) \big(|\tilde \eta_{2}^{\eps}(t)| + \eps^2)$ and
\begin{align*}
&	\Big|\frac{d}{dt}\eta_1^\eps(t) \Big| 
	\le \const(V,R,x_{0},\xi_{0},v_{\eps})
	\big( |\eta_1^\eps(t)| + |\tilde\eta_2^\eps(t)| + |\eta_3^\eps(t)| + \eps^2\big),  \\
&\Big|\frac{d}{dt}\tilde \eta_2^\eps(t) \Big| 
\le \const(V,R,x_{0},\xi_{0},v_{\eps})
\big( |\eta_1^\eps(t)| + |\tilde\eta_2^\eps(t)|+ |\eta_3^\eps(t)| + \eps^2 \big).
\end{align*}
Furthermore, by virtue of Lemma~\ref{w-come-in-k},
for every $t\in [0, T^{\eps,\gamma})$ we have
\begin{align*}
\Big|\frac{d}{dt}\eta_3^\eps(t) \Big| &\le 
\const(R,x_{0},\xi_{0},\delta) (|\eta_{1}^{\eps}(t)| + |\eta_{2}^{\eps}(t)|+ |\eta_{3}^{\eps}(t)| 
+\sqrt{\gamma} + \eps^{2}) \\
&\le 
\const(R,x_{0},\xi_{0},\delta) (|\eta_{1}^{\eps}(t)| + |\tilde \eta_{2}^{\eps}(t)|+ |\eta_{3}^{\eps}(t)| 
+\sqrt{\gamma} + \eps^{2}).
\end{align*}
It is readily verified that all the constants in the various estimates
contained in the previous sections can be bounded from above by 
quantities which are independent upon $\eps$. In turn, taking into account 
Lemma~\ref{eta-iniz} and Proposition \ref{viene-2}, there exists a positive constant $C$ such that
$$
|\eta_1^\eps(t)| + |\tilde\eta_2^\eps(t)|+|\eta_3^\eps(t)|\leq C\eps^2+C\int_0^t
(|\eta_1^\eps(\tau)| + |\tilde\eta_2^\eps(\tau)|+|\eta_{3}^{\eps}(\tau)| )d\tau,
\quad\text{for all $t\in [0,T^{\eps,\gamma})$}.
$$
Then, Gronwall lemma yields
$|\eta_1^\eps(t)| + |\tilde\eta_2^\eps(t)|+|\eta_{3}^{\eps}(t)| \leq C\eps^2$
for all $t\in [0,T^{\eps,\gamma})$ and in turn also
$|\eta_1^\eps(t)| + |\eta_2^\eps(t)|+|\eta_{3}^{\eps}(t)| \leq C\eps^2$
for all $t\in [0,T^{\eps,\gamma})$. Also from Lemma \ref{w-come-in-k}, it holds $\eps |w^{\eps}(t)| \leq C \eps^{2}$.
In particular in \eqref{t-eps-new} one can take $T^{\eps,\gamma}=T$ for $\eps$ small enough. Then, from Lemma~\ref{rappr} there exist functions $\theta^{\eps}:[0,T^{\eps,\gamma})\to[0,2\pi)$ such that
$$
\big\|u_\eps(t,x)-e^{\frac{\imath}{\eps}(\xi(t)\cdot x+\theta^{\eps}(t))}
R\Big(\frac{x-x^\eps(t)}{\eps}\Big)\big\|^2_{H^1_\eps}\leq C\eps^2,
\quad\text{for all $t\in [0,T]$}
$$
which together with 
$$
\big\| R\Big(\frac{x-x^\eps(t)}{\eps}\Big) - R\Big(\frac{x-x(t)}{\eps}\Big)\big\|^2_{H^1} \le |w^\eps|^2\, \| \nabla R\|^2_{H^1} \le 
C \eps^2
$$
concludes the proof of Theorem~\ref{main}.

\appendix

\section{Semi-singular potentials}
\label{semi-sing-sect}

Let $\eps,\delta\in (0,1]$, $N\geq 1$ and $0<p<2/N$. In this section, we shall consider the nonlinear Schr\"odinger equation
for a family of smooth nearly singular external potentials $V_\delta:\R^N\to\R$, 
\begin{equation}
	\label{problem-semis}
\im \partial_t u+\frac{\eps^2}{2}\Delta u-V_\delta(x)u+|u|^{2p}u=0,\qquad t\in\R,\,\, x\in\R^N,
\end{equation}
where $\im$ is the imaginary unit and $u:\R\times\R^N\to\CC$ is a complex-valued function. 
we want to investigate the soliton dynamics
behavior as $\eps\to 0$ of the solutions to~\eqref{problem-semis} which start from a rescaled bump-like
initial data of the form
\begin{equation}
	\label{initidata}
u(x,0)=R\Big(\frac{x-x_0}{\eps}\Big)e^{\frac{\im}{\eps} \xi_0\cdot x},
\qquad x_0,\xi_0\in\R^N,
\end{equation}
Consider, for each $\delta\in (0,1]$, the Newtonian system
\begin{equation}
\label{newton-semis}
\begin{cases}
\dot x=\xi, & \\ 
\dot\xi=-\nabla V_\delta(x),  & \\
x(0)=x_0,  \,\,\,
\xi(0)=\xi_0.
\end{cases}
\end{equation}
Under suitable assumptions on the potential $V_\delta$, for each $\delta\in (0,1]$,
system~\eqref{newton-semis} admits a unique global solution $t\mapsto (x_\delta(t),\xi_\delta(t))$
and its associated Hamiltonian energy 
${\mathscr H}_\delta(t)=\frac{1}{2}|\xi_\delta(t)|^2+V_\delta(x_\delta(t))$, $t\geq 0$,
remains constant through the motion.

\noindent
Let $(V_\delta)_{\delta\in(0,1]}$ be a family of functions $V_\delta\in C^3(\R^N,\R^+)$ such that $\|D^\alpha V_\delta\|_{L^\infty}<\infty$
for every $0\leq|\alpha|\leq 3$ and all $\delta\in (0,1]$. We define the function $\phi:(0,1]\to (0,\infty)$ by setting 
$$
\phi(\delta):=\sum_{0\leq |\alpha|\leq 3}\|D^\alpha V_\delta\|_{L^\infty},
$$
for all $\delta\in (0,1]$. We shall assume that there exists a 
set ${\mathscr V}\subset \R^+\times\R^+$ such that $(0,0)\in \bar{{\mathscr V}}$ 
and 
\begin{equation}
	\label{mainvanish}
\sup_{\substack{(\eps,\delta)\in {\mathscr V} \\ \eps,\delta\in (0,1]}}\,\eps^2\phi(\delta)<+\infty,
\qquad\quad
	\limsup_{\substack{(\eps,\delta)\in {\mathscr V} \\ \eps\to 0^+ \\ \delta\to 0^+}}\,\eps^2\phi(\delta)=0.
\end{equation}
Without loss of generality, we may assume that $\phi(\delta)\geq 1$, for all $\delta\in (0,1].$
\vskip5pt
\noindent
The main result of the Appendix, possibly useful for numerical purposes, is the following

\begin{theorem}
	\label{mainth}
	Let $T>0$ and let $u^{\eps,\delta}$ be the unique solution to problem \eqref{problem-semis}-\eqref{initidata}. 
	Assume \eqref{mainvanish} and that for the initial position $x_0\in\R^N$ it holds
	\begin{equation*}
	\sup_{\delta\in (0,1]}V_\delta(x_0)<+\infty.
	\end{equation*}
	Then there exist $C>0$, and $\eps_0,\delta_0>0$ sufficiently small that
	$$
	u^{\eps,\delta}(t,x)=R\left(\frac{\cdot-x_\delta(t)}{\eps}\right)e^{\frac{\im}{\eps}(\xi_\delta(t)\cdot x+\vartheta^{\eps,\delta}(t))}+E^{\eps,\delta}(x,t),
	\qquad \|E^{\eps,\delta}(t,\cdot)\|_{H^1_\eps}\leq C\eps\phi^{2}(\delta),
	$$
	uniformly on $[0,T]$ for all $(\eps,\delta)\in {\mathscr V}$ such that 
	$0<\eps\leq\eps_0$ and $0<\delta\leq\delta_0$, being $x_\delta(t)$ the solution to 
	system~\eqref{newton-semis} and $\vartheta^{\eps,\delta}$ a suitable shift term. In particular, provided that
	\begin{equation*}
		\limsup_{\substack{(\eps,\delta)\in {\mathscr V} \\ \eps\to 0^+ \\ \delta\to 0^+}}\,\eps\phi^{2}(\delta)=0,
	\end{equation*} 
	a soliton dynamic behavior occurs.
\end{theorem}

\noindent
The theorem will be proved using essentially the arguments developed in \cite{bronski,keraa} and explicitly 
highlighting the dependence of the conclusions from the parameter $\delta$ ruling the degree of singularity
of the potential.

\subsection{Preparatory results}

\noindent
It is known that the solution $u^{\eps,\delta}$ to \eqref{problem-semis}-\eqref{initidata} exists for all times $t$ with
$u^{\eps,\delta}(t)\in H^2(\R^N)$ and has conserved quantities,
the mass
\begin{equation}
	\label{masscons}
\frac{1}{\eps^N}\int_{\R^N}|u^{\eps,\delta}(t,x)|^2=\|R\|_{L^2}^2:=m
\end{equation}
independently of $\eps,\delta\in (0,1]$, and the energy
\begin{align*}
	E_{\eps,\delta}(t)=\frac{1}{2\eps^{N-2}}\int_{\R^N}|\nabla u^{\eps,\delta}(t)|^2
	+\frac{1}{\eps^{N}}\int_{\R^N}V_\delta(x)|u^{\eps,\delta}(t)|^2
	-\frac{1}{\eps^{N}(p+1)}\int_{\R^N}|u^{\eps,\delta}(t)|^{2p+2}=E_{\eps,\delta}(0).
\end{align*}
In the spirit of \cite[Lemma 3.3]{keraa} it holds
\begin{lemma}
	\label{control}
There exists a positive constant $C$ such that
$$
\left|\int_{\R^N}V_\delta(x_0+\eps x)R^2(x)-mV_\delta(x_0)\right|\leq C\eps^2\phi(\delta),\qquad
\forall x_0\in\R^N,\,\,\forall \eps\in (0,1],\,\,\forall \delta\in (0,1].
$$
\end{lemma}

\begin{lemma}
	\label{gradbound}
Let $u^{\eps,\delta}$ be the unique solution to \eqref{problem-semis}-\eqref{initidata}. Assume that for the initial position $x_0\in\R^N$
\begin{equation}
	\label{nonsingstart}
\sup_{\delta\in (0,1]}V_\delta(x_0)<+\infty.
\end{equation}
There exists a positive constant $C$ such that
$$
\sup_{t\geq 0}\|\nabla u^{\eps,\delta}(t)\|_{L^2}^2\leq C\eps^{N-2}+C\eps^{N}\phi(\delta),\qquad
\forall \eps\in (0,1],\,\,\forall \delta\in (0,1].
$$
In particular, in light of \eqref{mainvanish}, there holds
\begin{equation}
	\label{boundgrad-situation}
\sup_{(\eps,\delta)\in {\mathscr V}}\sup_{t\geq 0} \eps^{2-N}\|\nabla u^{\eps,\delta}(t)\|_{L^2}^2<+\infty.		
\end{equation}	
\end{lemma}
\begin{proof}
	Taking into account that $V_\delta(x)\geq 0$ for all $x\in\R^N$ and $\delta>0$, by the conservation of energy and 
	using Lemma~\ref{control} and assumption \eqref{nonsingstart}, it follows that
\begin{align*}
&	\frac{1}{2\eps^{N-2}}\int_{\R^N}|\nabla u^{\eps,\delta}(t)|^2
	-\frac{1}{\eps^{N}(p+1)}\int_{\R^N}|u^{\eps,\delta}(t)|^{2p+2}\leq E_{\eps,\delta}(0) \\
	& =\frac{1}{2}\int_{\R^N}|\nabla R|^2
	+\int_{\R^N}V_\delta(x_0+\eps x)R^2(x)
	-\frac{1}{p+1}\int_{\R^N}R^{2p+2}(x) \\
	& \leq\frac{1}{2}\int_{\R^N}|\nabla R|^2
	+mV_\delta(x_0)+C\eps^2\phi(\delta)
	-\frac{1}{p+1}\int_{\R^N}R^{2p+2}(x) 
	 \leq C+C\eps^2\phi(\delta),
\end{align*}
yielding in turn
$$
\|\nabla u^{\eps,\delta}(t)\|_{L^2}^2\leq C\eps^{N-2}+C\eps^N\phi(\delta)+\frac{C}{\eps^{2}}\|u^{\eps,\delta}(t)\|^{2p+2}_{2p+2},
$$
for all $t\in [0,\infty)$ and for any $\eps\in (0,1]$ and $\delta\in (0,1]$. Set $\theta=Np/(2p+2)$. By the conservation of
mass \eqref{masscons} it holds $\|u^{\eps,\delta}(t)\|_{L^2}=C\eps^{N/2}$ for all $t\in [0,\infty)$ and for any $\eps,\delta\in (0,1]$. 
Then, by virtue of the Gagliardo-Nirenberg inequality, it follows
$\|u^{\eps,\delta}(t)\|_{2p+2}\leq C\eps^{(1-\theta)N/2}\|\nabla u^{\eps,\delta}(t)\|_{L^2}^\theta$ for all $t\in [0,\infty)$ and any 
$\eps,\delta\in (0,1]$. By the definition of $\theta$ and Young's inequality we reach
$$
\frac{C}{\eps^{2}}\|u^{\eps,\delta}(t)\|^{2p+2}_{2p+2}\leq
C\eps^{N-2}+\frac{1}{2}\|\nabla u^{\eps,\delta}(t)\|_{L^2}^2,
$$ 
for all $t\geq 0,$ which immediately concludes the proof.
\end{proof}

\noindent
The solution $u^{\eps,\delta}$ enjoys the following energy splitting.

\begin{lemma}
	\label{split1}
Let $u^{\eps,\delta}$ be the unique solution to \eqref{problem-semis}-\eqref{initidata}.
There exists a positive constant $C$ such that
\begin{equation*}
E_{\eps,\delta}(u^{\eps,\delta}(t))={\mathscr E}(R)+m{\mathscr H}_\delta(t)+C\eps^2\phi(\delta),
\end{equation*}
for all $t\in [0,\infty)$ and for any $\eps\in (0,1]$ and $\delta\in (0,1]$. 
\end{lemma}
\begin{proof}
	It is sufficient to observe that, by the conservation of energies $E_{\eps,\delta}$ and ${\mathscr H}_\delta$ and 
	taking into account Lemma~\ref{control}, we obtain
	\begin{equation*}
	E_{\eps,\delta}(u^{\eps,\delta}(t))
	= \frac{1}{2}m |\xi_0|^2+mV_\delta(x_0)+{\mathscr E}(R) + C\eps^2\phi(\delta) 
=m{\mathscr H}_\delta(t)   
	+{\mathscr E}(R)+ C\eps^2\phi(\delta),
\end{equation*}
for all $t\in [0,\infty)$ and for any $\eps\in (0,1]$ and $\delta\in (0,1]$. 
\end{proof}

\noindent
Following \cite{keraa}, let us now consider the auxiliary function
\begin{equation}
	\label{auxiliary}
\Psi_{\eps,\delta}(t,x):=e^{-\frac{\im}{\eps}(\eps x+x_\delta(t))\cdot\xi_\delta(t)}u^{\eps,\delta}(\eps x+x_\delta(t)).
\end{equation}
It is readily seen that $\|\Psi_{\eps,\delta}(t)\|_{L^2}^2=\|R\|_{L^2}^2$ for every $t\geq 0$ and 
\begin{equation}
	\label{es-identity}
{\mathscr E}(\Psi_{\eps,\delta}(t))=E_{\eps,\delta}(u^{\eps,\delta}(t))
-\frac{1}{\eps^N}\int_{\R^N}V_\delta(x)|u^{\eps,\delta}|^2+\frac{1}{2}m|\xi_\delta(t)|^2-\xi_\delta(t)\cdot\int_{\R^N}p_{\eps,\delta}(t,x),
\end{equation}
where $p_{\eps,\delta}:\R\times\R^N\to\R^N$ is the momentum defined by
$$
p_{\eps,\delta}(t,x):=\frac{1}{\eps^{N-1}}\Im(\bar u^{\eps,\delta}(t,x)\nabla u^{\eps,\delta}(t,x)),\qquad t\in\R,\,\, x\in\R^N.
$$

\begin{lemma}
	\label{mombound}
	There exists a positive constant $C$ such that
	$$
	\left|\int_{\R^N}p_{\eps,\delta}(t,x)\right|\leq C+C\eps\sqrt{\phi(\delta)},\qquad 
	\forall t\in [0,\infty),\,\,\,\forall \eps\in (0,1],\,\,\forall \delta\in (0,1].
	$$
	In particular, in light of \eqref{mainvanish}, there holds
	\begin{equation}
		\label{boundmom-situation}
	\sup_{(\eps,\delta)\in {\mathscr V}}\sup_{t\geq 0} \left|\int_{\R^N}p_{\eps,\delta}(t,x)\right|  <+\infty.		
	\end{equation}
\end{lemma}
\begin{proof}
Taking into account \eqref{masscons}, by H\"older inequality we get
\begin{equation*}
\left|\int_{\R^N}p_{\eps,\delta}(t,x)\right|\leq \int_{\R^N} 
\frac{|u^{\eps,\delta}(t,x)|}{\eps^{N/2}}\frac{|\nabla u^{\eps,\delta}(t,x)|}{\eps^{N/2-1}}	
\leq C\eps^{\frac{2-N}{2}}\|\nabla u^{\eps,\delta}(t)\|_{L^2}\leq C+C\eps\sqrt{\phi(\delta)}.
\end{equation*}
for all $t\in [0,\infty)$ and for any $\eps\in (0,1]$ and $\delta\in (0,1]$. 
Then, Lemma~\ref{gradbound} yields the assertion.
\end{proof}

\begin{lemma}
	\label{controlODE}
Assume that \eqref{nonsingstart} holds. Then, there holds
$$
\sup_{\delta\in (0,1]}\sup_{t\geq 0}|\xi_\delta(t)|<+\infty.
$$
\end{lemma}
\begin{proof}
Since the energy functional ${\mathscr H}_\delta(t)=\frac{1}{2}|\xi_\delta(t)|^2+V_\delta(x_\delta(t))$
associated with~\eqref{newton-semis} remains constant, for any $t\geq 0$, taking into account that $V_\delta\geq 0$, there holds
$$
|\xi_\delta(t)|^2=2{\mathscr H}_\delta(t)-2V_\delta(x_\delta(t))\leq 2{\mathscr H}_\delta(t)=2{\mathscr H}_\delta(0)
=|\xi_0|^2+V_\delta(x_0)\leq C,
$$
where the last bound is due to \eqref{nonsingstart}. This proves the assertion.
\end{proof}

\begin{lemma}
	\label{controlODE2}
Assume that \eqref{nonsingstart} holds. Then
$$
\sup_{\delta\in (0,1]}\sup_{t\in [0,\phi(\delta)^{-1}]}|x_\delta(t)|<+\infty.
$$
\end{lemma}
\begin{proof}
In light of Lemma~\ref{controlODE} and since $\phi(\delta)\geq 1$ there holds
$$
|x_\delta(t)|\leq |x_0|+\int_0^t |\xi_\delta(s)|ds\leq C+Ct\leq C+\frac{C}{\phi(\delta)}\leq C,
$$
for all $\delta\in (0,1]$ and $t\in [0,\phi(\delta)^{-1}]$, yielding the assertion.
\end{proof}

\noindent
We now recall \cite[Lemma 3.2]{keraa} the following
\begin{lemma}
	\label{deltalemm}
There exist $C_0>1$ and $K_0>0$ with $|\xi_2-\xi_1|\leq C_0\|\delta_{\xi_2}-\delta_{\xi_1}\|_{C^{2*}}$ if 
$\|\delta_{\xi_2}-\delta_{\xi_1}\|_{C^{2*}}\leq K_0$.	
\end{lemma}

\noindent
On account of \eqref{es-identity} and Lemma~\ref{split1}, for the family $\Psi_{\eps,\delta}$ we have the following energy splitting
\begin{equation*}
{\mathcal E}(\Psi^{\eps,\delta}(t))-{\mathscr E}(R)=\xi_\delta(t)\cdot\Big(
m\xi_\delta(t)-\int_{\R^N}p_{\eps,\delta}(t,x)\Big)
+mV_\delta(x_\delta(t))-\frac{1}{\eps^N}\int_{\R^N}V_\delta(x)|u^{\eps,\delta}(t,x)|^2
+C\eps^2\phi(\delta),
\end{equation*}
for all $t\in [0,\infty)$ and for any $\eps\in (0,1]$ and $\delta\in (0,1]$. 
We shall now set
$$
\eta_1^{\eps,\delta}(t):=m\xi_\delta(t)-\int_{\R^N}p_{\eps,\delta}(t,x),\qquad
\eta_2^{\eps,\delta}(t):=mV_\delta(x_\delta(t))-\frac{1}{\eps^N}\int_{\R^N}V_\delta(x)|u^{\eps,\delta}(t,x)|^2,
$$
Furthermore, if $C_0,K_0$ are the constants in Lemma~\ref{deltalemm}, let us set
$$
M:=\sup_{\delta \in (0,1]}\sup_{t\in [0,\phi(\delta)^{-1}]} C_0|x_\delta(t)|+C_0K_0.
$$
In light of Lemma~\ref{controlODE2}, $M>0$ is finite. Of course $|x_\delta(t)|\leq M$, for 
every $\delta\in (0,1]$ and $t\in [0,\phi(\delta)^{-1}]$. We denote by $\chi$ a cut-off
function such that $\chi=1$ on $|x|\leq M$ and $\chi=0$ on $|x|\geq 2M$. Finally, also set
$$
\eta_3^{\eps,\delta}(t):=m x_\delta(t)-\frac{1}{\eps^N}\int_{\R^N}x\chi(x)|u^{\eps,\delta}(t,x)|^2,
$$
for every $t\geq 0$.
Taking into account Lemma~\ref{controlODE}, we finally achieve the following

\begin{lemma}
	\label{split2}
	Let $u^{\eps,\delta}$ be the unique solution to problem \eqref{problem-semis}-\eqref{initidata} and
	let $\Psi^{\eps,\delta}$ the function defined in formula \eqref{auxiliary}.
	Furthermore, let us set $\eta^{\eps,\delta}(t)=|\eta_1^{\eps,\delta}(t)|+|\eta_2^{\eps,\delta}(t)|
	+|\eta_3^{\eps,\delta}(t)|$. Then there exists a positive constant $C$ such that	
$0\leq {\mathcal E}(\Psi^{\eps,\delta}(t))-{\mathscr E}(R)\leq C\eta^{\eps,\delta}(t)+C\eps^2\phi(\delta)$,
for every $t\geq 0$.
\end{lemma}

\begin{lemma}
The functions $\eta_i^{\eps,\delta}:[0,\infty)\to\R$, $i=1,2,3$ are continuous and 
$$
\eta_1^{\eps,\delta}(0)=0,
\qquad
|\eta_2^{\eps,\delta}(0)|\leq C\eps^2\phi(\delta),
\qquad
|\eta_3^{\eps,\delta}(0)|\leq C\eps^2,
$$
for some $C>0$ and for any $\eps\in (0,1]$ and $\delta\in (0,1]$.
\end{lemma}
\begin{proof}
We easily get $\eta_1^{\eps,\delta}(0)=m\xi_0-\int_{\R^N}p_{\eps,\delta}(0,x)=m\xi_0-\xi_0\int_{\R^N}R^2(x)=0.$
Moreover, we have
\begin{equation*}
|\eta_2^{\eps,\delta}(0)| =\Big|mV_\delta(x_0)-\int_{\R^N}V_\delta(x_0+\eps x)R^2\Big|\leq C\eps^2\phi(\delta), \,\,
|\eta_3^{\eps,\delta}(0)| =\Big|m x_0-\int_{\R^N}(x_0+\eps x)\chi(x_0+\eps x)R^2\Big|\leq C\eps^2,
\end{equation*}
in light of Lemma~\ref{control}.
\end{proof}

\noindent
Let us introduce the time
\begin{equation}
	\label{TTimedef}
T^{\eps,\delta}:=\sup\{t\in \big[0,\phi(\delta)^{-1}\big]: \eta^{\eps,\delta}(s)\leq \mu,\,\,\text{for all $s\in (0,t)$}\},
\end{equation}
where, recalling \eqref{mainvanish}, $\mu>0$ is a sufficiently small positive constant such that 
\begin{equation}
	\label{TTimedef1}
\begin{cases}
C\eta^{\eps,\delta}(t)+C\eps^2\phi(\delta)\leq {\mathcal A},   & \\
\noalign{\vskip3pt}
\text{for every $t\in [0,T^{\eps,\delta})$ and all $(\eps,\delta)\in {\mathscr V}$ such that} & \\
\noalign{\vskip3pt}
\text{$0<\eps\leq\eps_0$ and $0<\delta\leq\delta_0$, where $\eps_0,\delta_0$ are small enough.}
\end{cases}
\end{equation}
being ${\mathcal A}>0$ the constant which appears in the statement of Proposition~\ref{weinstein} and $C>0$
is the constant which appears in the statement of Lemma~\ref{split2}.
\vskip4pt
\noindent
In this framework, by virtue of Proposition~\ref{weinstein}, we find
families of functions $\theta^{\eps,\delta}:[0,T^{\eps,\delta})\to[0,2\pi)$
and $\xi^{\eps,\delta}:[0,T^{\eps,\delta})\to\R^N$ such that
$$
\left\|\Psi^{\eps,\delta}(t)-e^{\im \theta^{\eps,\delta}}R(\cdot-\xi^{\eps,\delta})\right\|^2_{H^1}\leq C\eta^{\eps,\delta}(t)+C\eps^2\phi(\delta),
$$
for all $t\in [0,T^{\eps,\delta})$. Then, 
we get the following

\begin{lemma}
	\label{rappr-semis}
There exist families of functions $\theta^{\eps,\delta}:[0,T^{\eps,\delta})\to[0,2\pi)$
and $\xi^{\eps,\delta}:[0,T^{\eps,\delta})\to\R^N$ with
$$
\left\|u^{\eps,\delta}(t)-e^{\frac{\im}{\eps}(\xi_\delta(t)\cdot x+\vartheta^{\eps,\delta}(t))}
R\left(\frac{\cdot-x_\delta(t)}{\eps}+\xi^{\eps,\delta}\right)\right\|^2_{H^1_\eps}\leq C\eta^{\eps,\delta}(t)+C\eps^2\phi(\delta),
$$
for all $t\in [0,T^{\eps,\delta})$,
where $w^{\eps,\delta}:=x_\delta(t)-\eps \xi^{\eps,\delta}$ and $\vartheta^{\eps,\delta}(t):=\eps \theta^{\eps,\delta}(t)$.
\end{lemma}

\noindent
We now aim to prove the following

\begin{lemma}
	\label{dual-estimate}
Let $u^{\eps,\delta}$ be the unique solution to \eqref{problem-semis}-\eqref{initidata}. Then
there exists a positive constant $C$ with
$$
\|\eps^{-N}u^{\eps,\delta}(\cdot,t)-m\delta_{x_\delta(t)}\|_{C^{2*}}+
\|p^{\eps,\delta}(\cdot,t)dx-m\xi_\delta(t)\delta_{x_\delta(t)}\|_{C^{2*}}
\leq C\eta^{\eps,\delta}(t)+C\eps^2\phi(\delta),
$$
for every $t\in [0,T^{\eps,\delta})$ and all $(\eps,\delta)\in {\mathscr V}$ such that 
$0<\eps\leq\eps_0$ and $0<\delta\leq\delta_0$.
\end{lemma}
\begin{proof}
Let $u^{\eps,\delta}$ be the unique solution to problem \eqref{problem-semis}-\eqref{initidata}. 
Then, in the spirit of \cite[Lemma 3.4]{keraa}, it is possible to prove	
that there exists a positive constant $C$, independent of $\eps$ and $\delta$, such that
\begin{equation}
	\label{primostep}
\|\eps^{-N}u^{\eps,\delta}(\cdot,t)dx-m\delta_{w^{\eps,\delta}(t)}\|_{C^{2*}}+
\|p^{\eps,\delta}(\cdot,t)dx-m\xi_\delta(t)\delta_{w^{\eps,\delta}(t)}\|_{C^{2*}}
\leq C\eta^{\eps,\delta}(t)+C\eps^2\phi(\delta),
\end{equation}
for every $t\in [0,T^{\eps,\delta})$ and all $(\eps,\delta)\in {\mathscr V}$ such that 
$0<\eps\leq\eps_0$ and $0<\delta\leq\delta_0$.	
Let us now prove that there exists $\mu>0$ and a positive constant $C$,
independent of $\eps$ and $\delta$, such that
\begin{equation}
	\label{pointcontro}
|x_\delta(t)-w^{\eps,\delta}(t)|\leq C\eta^{\eps,\delta}(t)+C\eps^2\phi(\delta),
\end{equation}
for every $t\in [0,T^{\eps,\delta})$ and all $(\eps,\delta)\in {\mathscr V}$ such that 
$0<\eps\leq\eps_0$ and $0<\delta\leq\delta_0$.
We follow the proof of \cite[Lemma 3.5]{keraa}. Assuming that $|w^{\eps,\delta}(t)|\leq M$
for every $t\in [0,T^{\eps,\delta})$ and all $(\eps,\delta)\in {\mathscr V}$ such that 
$0<\eps\leq\eps_0$ and $0<\delta\leq\delta_0$ (up to further reducing the size of $\delta_0$), 
the assertion follows, since by the definition of $\chi$ and \eqref{primostep},
\begin{align}
|x_\delta(t)-w^{\eps,\delta}(t)| & \leq \frac{1}{m}\Big|\int_{\R^N}x\chi(x)\frac{|u^{\eps,\delta}(t,x)|^2}{\eps^N}-m w^{\eps,\delta}(t)   \Big|+\frac{1}{m}\eta^{\eps,\delta}(t)  \label{perprimap}\\
& \leq C\|x\chi\|_{C^2}\|\eps^{-N}u^{\eps,\delta}(\cdot,t)dx-
m\delta_{w^{\eps,\delta}(t)}\|_{C^{2*}}+C\eta^{\eps,\delta}(t) 
\leq C\eta^{\eps,\delta}(t)+C\eps^2\phi(\delta), \notag
\end{align}
for every $t\in [0,T^{\eps,\delta})$ and all $(\eps,\delta)\in {\mathscr V}$ such that 
$0<\eps\leq\eps_0$ and $0<\delta\leq\delta_0$. Thus, it is left to show that $|w^{\eps,\delta}(t)|\leq M$
for all $t\in [0,T^{\eps,\delta})$ and $(\eps,\delta)\in {\mathscr V}$ such that 
$0<\eps\leq\eps_0$ and $0<\delta\leq\delta_0$, up to further reducing the size of $\delta_0$.
On account of Lemma~\ref{mombound}, and arguing as in \cite[p.183]{keraa} there exists a constant
$C$, independent of $\eps$ and $\delta$, such that for all $t_1,t_2\in [0,T^{\eps,\delta})$ with $t_1<t_2$
$$
\|\eps^{-N}u^{\eps,\delta}(\cdot,t_2)dx-\eps^{-N}u^{\eps,\delta}(\cdot,t_1)dx\|_{C^{2*}}\leq C|t_2-t_1|\leq \frac{2C}{\phi(\delta)},
$$
yielding in turn by \eqref{primostep} and the definition \eqref{TTimedef}-\eqref{TTimedef1} of $T^{\eps,\delta}$
\begin{equation*}
\|m\delta_{w^{\eps,\delta}(t_2)}-m\delta_{w^{\eps,\delta}(t_1)}\|_{C^{2*}}\leq 
C\Big[\eta^{\eps,\delta}(t_2)+\eta^{\eps,\delta}(t_1)+\eps^2\phi(\delta)+\frac{1}{\phi(\delta)}\Big]  
\leq C\mu+C\eps^2\phi(\delta)+\frac{C}{\phi(\delta)}.
\end{equation*}
Therefore, up to reducing the value of $\delta_0$, choosing $\mu>0$ sufficiently small
in the definition of $T^{\eps,\delta}$,
we have $\|\delta_{w^{\eps,\delta}(t_2)}-\delta_{w^{\eps,\delta}(t_1)}\|_{C^{2*}}\leq K_0$
for every $t\in [0,T^{\eps,\delta})$ and all $(\eps,\delta)\in {\mathscr V}$ such that 
$0<\eps\leq\eps_0$ and $0<\delta\leq\delta_0$, where $K_0$ is the constant appearing in 
the statement of Lemma~\ref{deltalemm}. By virtue of Lemma~\ref{deltalemm}, it holds
$|w^{\eps,\delta}(t_2)-w^{\eps,\delta}(t_1)|\leq C_0
\|\delta_{w^{\eps,\delta}(t_2)}-\delta_{w^{\eps,\delta}(t_1)}\|_{C^{2*}}\leq C_0K_0$. Since $w^{\eps,\delta}(0)=x_0$,
it follows $|w^{\eps,\delta}(t)|\leq C_0K_0+|x_0|\leq M$, yielding the desired conclusion.
As a consequence of \eqref{pointcontro}, there holds 
$\|\delta_{x_\delta(t)}-\delta_{w^{\eps,\delta}(t)}\|_{C^{2*}}
\leq |x_\delta(t)-w^{\eps,\delta}(t)|\leq C\eta^{\eps,\delta}(t)+C\eps^2\phi(\delta),$ 
for all $t\in [0,T^{\eps,\delta})$ and all $(\eps,\delta)\in {\mathscr V}$ such that 
$0<\eps\leq\eps_0$ and $0<\delta\leq\delta_0$. This yields the assertion by \eqref{primostep}. 
\end{proof}

\begin{lemma}
	\label{gronw}
$\eta^{\eps,\delta}(t)\leq C\eps^2\phi(\delta)$
for all $t\in [0,T^{\eps,\delta})$ and all $(\eps,\delta)\in {\mathscr V}$ with
$0<\eps\leq\eps_0$ and $0<\delta\leq\delta_0$.
\end{lemma}
\begin{proof}
	We have
	$$
	\eta^{\eps,\delta}(t)\leq C\eps^2\phi(\delta)+\int_0^t \sum_{j=1}^3\left|\frac{d}{dt}\eta^{\eps,\delta}_j(\sigma)\right|d\sigma
	$$
We recall that, as known, the following identities holds
$$
\int_{\R^N}\frac{\partial }{\partial t}p^{\eps,\delta}(t,x)=\frac{1}{\eps^N}\int_{\R^N}\nabla V_\delta (x)|u^{\eps,\delta}(t,x)|^2,\qquad\quad
\frac{\partial }{\partial t}\frac{|u^{\eps,\delta}(t,x)|^2}{\eps^N}=-{\rm div}_x p^{\eps,\delta}(t,x).
$$	
In turn, by Lemma~\ref{dual-estimate}, we have
\begin{align*}
\left|\frac{d}{dt}\eta_1^{\eps,\delta}(t)\right| &=\Big|m\dot\xi_\delta(t)+\frac{1}{\eps^N}\int_{\R^N}\nabla V_\delta (x)|u^{\eps,\delta}(t,x)|^2 \Big|\\
& \leq \|\nabla V_\delta\|_{C^2}\|\|\eps^{-N}u^{\eps,\delta}(\cdot,t)dx-m\delta_{x_\delta(t)}\|_{C^{2*}} 
\leq C\phi(\delta)\eta^{\eps,\delta}(t)+C\eps^2\phi^2(\delta),
\end{align*}
for every $t\in [0,T^{\eps,\delta})$ and all $(\eps,\delta)\in {\mathscr V}$ such that 
$0<\eps\leq\eps_0$ and $0<\delta\leq\delta_0$. Then, 
$$
\int_0^t\left|\frac{d}{dt}\eta^{\eps,\delta}_1(\sigma)\right|d\sigma
\leq C\phi(\delta)\int_0^t\eta^{\eps,\delta}(\sigma)d\sigma+\int_0^t C\eps^2\phi^2(\delta)
\leq C\phi(\delta)\int_0^t\eta^{\eps,\delta}(\sigma)d\sigma+C\eps^2\phi(\delta),
$$
since $t\leq T^{\eps,\delta}\leq \phi(\delta)^{-1}$. Analogously, again by Lemma~\ref{dual-estimate}, we have
\begin{align*}
\left|\frac{d}{dt}\eta_2^{\eps,\delta}(t)\right| 
 &=\left|\nabla V_\delta(x_\delta(t))\cdot m\xi_\delta(t)-\int_{\R^N}\nabla V_\delta(x)\cdot p^{\eps,\delta}(t,x)\right| \\
& \leq \|\nabla V_\delta\|_{C^2}\|\|p^{\eps,\delta}(\cdot,t)dx-m\xi_\delta(t)\delta_{x_\delta(t)}\|_{C^{2*}} 
\leq C\phi(\delta)\eta^{\eps,\delta}(t)+C\eps^2\phi^2(\delta).
\end{align*}
 Then, as $t\leq T^{\eps,\delta}\leq \phi(\delta)^{-1}$, we achieve
$$
\int_0^t\left|\frac{d}{dt}\eta^{\eps,\delta}_2(\sigma)\right|d\sigma
\leq C\phi(\delta)\int_0^t\eta^{\eps,\delta}(\sigma)d\sigma+\int_0^t C\eps^2\phi^2(\delta)
\leq C\phi(\delta)\int_0^t\eta^{\eps,\delta}(\sigma)d\sigma+C\eps^2\phi(\delta).
$$
The treatment of the term $\eta^{\eps,\delta}_3$ follows as in the proof of \cite[Lemma 3.6]{keraa} yielding, as $t\leq T^{\eps,\delta}\leq \phi(\delta)^{-1}$,
\begin{equation*}
\int_0^t\left|\frac{d}{dt}\eta_3^{\eps,\delta}(\sigma)\right|d\sigma
\leq C\int_0^t\eta^{\eps,\delta}(\sigma)d\sigma+C\eps^2
\leq C\phi(\delta)\int_0^t\eta^{\eps,\delta}(\sigma)d\sigma+C\eps^2\phi(\delta).
\end{equation*}
Hence, by recollecting the previous inequalities, by virtue of Gronwall lemma and $t\leq T^{\eps,\delta}\leq \phi(\delta)^{-1}$, it follows
$\eta^{\eps,\delta}(t)\leq C\eps^2\phi(\delta)e^{\phi(\delta)t}\leq C\eps^2\phi(\delta)$,
concluding the proof.
\end{proof}

\subsection{Proof of Theorem~\ref{mainth}}

By Lemma~\ref{gronw} and the continuity of $\eta^{\eps,\delta}$, it follows 
$T^{\eps,\delta}=\phi(\delta)^{-1}$, yielding
$\eta^{\eps,\delta}(t)\leq C\eps^2\phi(\delta),$
for every $t\in [0,\phi(\delta)^{-1})$ and all $(\eps,\delta)\in {\mathscr V}$ such that 
$0<\eps\leq\eps_0$ and $0<\delta\leq\delta_0$,
up to reducing the value of $\eps_0$ and $\delta_0$. Hence,
\begin{equation}
	\label{eq-quasi-fin}
\Big\|u^{\eps,\delta}(t)-e^{\frac{\im}{\eps}(\xi_\delta(t)\cdot x+\vartheta^{\eps,\delta}(t))}
R\left(\frac{\cdot-x_\delta(t)}{\eps}+\xi^{\eps,\delta}\right)
\Big\|^2_{H^1_\eps}\leq C\eps^2\phi(\delta),
\end{equation}
for all $t\in [0,\phi(\delta)^{-1})$ and all $(\eps,\delta)\in {\mathscr V}$ with 
$0<\eps\leq\eps_0$ and $0<\delta\leq\delta_0$. Recall now that, since $w^{\eps,\delta}=x_\delta(t)-\eps \xi^{\eps,\delta}$,
in light of \eqref{pointcontro}, we obtain $|\xi^{\eps,\delta}|^2\leq C\eps^2\phi(\delta)^2$. Then, we can conclude that
$\|R(\cdot)-R(\cdot-\xi^{\eps,\delta})\|_{H^1}^2\leq C|\xi^{\eps,\delta}|^2\leq C\eps^2\phi(\delta)^2$. This combined 
with \eqref{eq-quasi-fin} yields
\begin{equation}
	\label{eq-quasi-fin-2}
\Big\|u^{\eps,\delta}(t)-e^{\frac{\im}{\eps}(\xi_\delta(t)\cdot x+\vartheta^{\eps,\delta}(t))}
R\left(\frac{\cdot-x_\delta(t)}{\eps}\right)
\Big\|_{H^1_\eps}\leq C\eps\phi(\delta),
\end{equation}
for all $t\in [0,\phi(\delta)^{-1})$ and all $(\eps,\delta)\in {\mathscr V}$ with 
$0<\eps\leq\eps_0$ and $0<\delta\leq\delta_0$. Fixed $T>0$ and arguing as in \cite{keraa}, up to an
error of size $\eps\phi(\delta)$ in $H^1_\eps$ one can repeat the argument on the time interval
$[\phi(\delta)^{-1},2\phi(\delta)^{-1}]$ and so on. To cover the entire interval $[0,T]$ one therefore needs
to add $\phi(\delta)$-times an error of size $\eps\phi(\delta)$ in $H^1_\eps$, yielding
an overall error $\eps\phi^2(\delta)$ in $H^1_\eps$, reaching the control
\begin{equation}
	\label{eq-quasi-fin-3}
\Big\|u^{\eps,\delta}(t)-e^{\frac{\im}{\eps}(\xi_\delta(t)\cdot x+\vartheta^{\eps,\delta}(t))}
R\left(\frac{\cdot-x_\delta(t)}{\eps}\right)
\Big\|_{H^1_\eps}\leq C\eps\phi^2(\delta),
\end{equation}
for all $t\in [0,T]$ and all $(\eps,\delta)\in {\mathscr V}$ with 
$0<\eps\leq\eps_0$ and $0<\delta\leq\delta_0$. This concludes the proof.

\end{document}